\documentclass[english,12pt,a4paper]{amsart}
\usepackage[utf8]{inputenc}
\usepackage{a4wide}
\usepackage{esint}
\usepackage{hyperref}
\usepackage{mathtools}

\usepackage{amssymb,amsmath,amsthm,amscd}
\usepackage{mathrsfs}
\usepackage{tikz}
\usepackage{tikz-cd}
\usetikzlibrary{positioning}
\usepackage{graphics}
\usetikzlibrary{matrix,arrows}
\usepackage{amssymb}
\usepackage{babel}
\usepackage{amstext}
\usepackage{amscd}
\usepackage{epsfig}
\usepackage{rotating}
\usepackage{psfrag}
\usepackage{multicol}
\usepackage{multirow}
\usepackage{array}
\usepackage{verbatim}
\usepackage{enumitem}
\usepackage{colortbl}
\usepackage{hhline}
\usepackage{xcolor}
\usepackage[abs]{overpic}
\usepackage{minibox}
\usepackage{enumitem}
\usepackage{mathtools}

\makeatletter
\numberwithin{equation}{section}
\numberwithin{figure}{section}

\theoremstyle{plain}
\newtheorem{thm}{Theorem}
  \theoremstyle{plain}
  \numberwithin{thm}{section}
  \newtheorem{cor}[thm]{Corollary}
  \theoremstyle{plain}
  \newtheorem{lem}[thm]{Lemma}
  \theoremstyle{remark}
  \newtheorem{rem}[thm]{Remark}
    \theoremstyle{remark}
  \newtheorem{example}[thm]{Example}
    
   \theoremstyle{plain}
  \newtheorem{prop}[thm]{Proposition}
  \newtheorem{definition}[thm]{Definition}
  \def\Ddots{\mathinner{\mkern1mu\raise\p@
\vbox{\kern7\p@\hbox{.}}\mkern2mu
\raise4\p@\hbox{.}\mkern2mu\raise7\p@\hbox{.}\mkern1mu}}
\makeatother

\usepackage{babel}

\newcounter{anhang}

\newcommand{\mklm}[1]{\left\{ #1 \right\}}
\newcommand{\eklm}[1]{\left\langle #1 \right\rangle}

\renewcommand{\d}{\,d}

\newcommand{\N}{{\mathbb N}}
\newcommand{\Z}{{\mathbb Z}}
\newcommand{\C}{{\mathbb C}}

\newcommand{\R}{{\mathbb R}}

\newcommand{\E}{{\mathcal E}}
\newcommand{\F}{{\mathcal F}}
\newcommand{\Fbb}{{\mathbb F}}

\newcommand{\J}{{ \mathrm J}}
\newcommand{\M}{{\mathscr M}}

\newcommand{\CP}{{\mathbb{C}\mathrm{P}}}

\renewcommand{\epsilon}{\varepsilon}

\renewcommand{\rho}{\varrho}

\newcommand{\bdm}{\begin{displaymath}}
\newcommand{\edm}{\end{displaymath}}
\newcommand{\bq}{\begin{equation}}
\newcommand{\eq}{\end{equation}}
\newcommand{\bqn}{\begin{equation*}}
\newcommand{\eqn}{\end{equation*}}

\newcommand{\G}{{\mathcal G}}

\newcommand{\SO}{\mathrm{SO}}

\newcommand{\g}{{\bf \mathfrak g}}

\newcommand{\U}{{\mathrm{U}}}

\newcommand{\eps}{\varepsilon}

\newcommand{\Ad}{\mathrm{Ad}\,}

\newcommand{\id}{\mathrm{id}\,}

\renewcommand{\Im}{\mathrm{Im}\,}

\newcommand{\pr}{\mathrm{pr}}

\newcommand{\BLG}{\mathrm{Bl}_{{G}}}
\newcommand{\BLS}{\mathrm{Bl}_{{S^1}}}

\DeclareMathOperator{\Car}{Car}
\DeclareMathOperator{\supp}{supp\,}

\DeclareMathOperator{\codim}{codim}

\DeclareMathOperator{\coker}{coker}

\begin{document}
\title[Singular cohomology of  symplectic quotients and Kirwan surjectivity ]{Singular cohomology of  symplectic quotients  by circle actions and Kirwan surjectivity  }
\author{Benjamin Delarue}
\address{Universit\"at Paderborn, Warburger Str. 100, 33098 Paderborn, Germany}
\email{bdelarue@math.uni-paderborn.de}
\author{Pablo Ramacher}\address{Fachbereich 12 Mathematik und Informatik, Philipps-Universit\"at Marburg, Hans--Meer\-wein-Str., 35032 Marburg, Germany}
\email{ramacher@mathematik.uni-marburg.de}
\author{Maximilian Schmitt}
\address{Max Planck Institute for the Study of Crime, Security and Law, Günterstalstraße 73, 79100 Freiburg, Germany}
\email{m.schmitt@csl.mpg.de}

\date{December 6, 2023}

\begin{abstract}
Let $M$ be a symplectic manifold carrying a Hamiltonian $S^1$-action with  momentum map $\J:M \rightarrow \R$ and consider   the corresponding symplectic quotient $\M_0:=\J^{-1}(0)/S^1$. We extend Sjamaar's complex of differential forms on $\M_0$, whose cohomology is isomorphic to the singular or \v Cech cohomology $H(\M_0;\R)$ of $\M_0$ with real coefficients,  to a complex of differential forms  on $\M_0$ associated with a partial desingularization $\widetilde \M_0$, which we call resolution differential forms. The cohomology of that complex  turns out to be isomorphic to the de Rham cohomology $H(\widetilde \M_0)$ of $\widetilde \M_0$. Based on this, we derive a long exact sequence involving both  $H(\M_0;\R)$ and $H(\widetilde \M_0)$ and give conditions for its  splitting. We then define a Kirwan map $\mathcal{K}:H_{S^1}(M) \rightarrow H(\widetilde \M_0)$ from the equivariant cohomology $H_{S^1}(M)$ of $M$ to $H(\widetilde \M_0)$ and show that its image contains the image of $H(\M_0;\R)$ in $H(\widetilde \M_0)$ under the natural inclusion. Combining both results in the case that all fixed point components of $M$ have vanishing odd cohomology we obtain a surjection $\check \kappa:H^\textrm{ev}_{S^1}(M) \rightarrow H^\textrm{ev}(\M_0;\R)$ in even degrees, while already simple examples show  that a similar surjection in odd degrees does not exist in general. As an interesting class of examples we study abelian polygon spaces. 
\end{abstract}

\maketitle

\tableofcontents{}

\section{Introduction}\label{sec:1}

Let a compact Lie group $G$ act on a compact connected symplectic manifold $(M,\sigma)$ in a Hamiltonian fashion with equivariant momentum map $\J \colon M \rightarrow \mathfrak{g}^{*}$. Consider the associated \emph{symplectic quotient}  or \emph{reduced space}, which is given by the connected topological space
\[
\M_0:=\J^{-1}(0)/G.
\]
 If $0 \in \mathfrak{g}^{*}$ is a regular value of the momentum map, $M_{0}$ is a symplectic orbifold. But if $0$ is a singular value,  the symplectic quotient can have more serious singularities and is a stratified space whose strata are symplectic manifolds  induced by the stratification of $\J^{-1}(0)$ by orbit types. In spite of the fact that the  geometry and topology  of $\M_0$  have been extensively studied in the last decades by Kirwan, Lerman, Sjamaar, Witten, Kalkman, Tolman, Jeffrey, Kiem, Woolf, and others \cite{kirwan84,lerman-sjamaar,witten92,kalkman95,lerman-tolman00,JKKW03}, the cohomology of $\M_0$ as a stratified space  is still not fully understood. In particular,  its ordinary cohomology $H^\ast(\M_0;\R)$, that is, its singular or \v Cech cohomology with coefficients in $\R$, in the following also simply denoted by $H^\ast(\M_0)$, and its relation with the intersection cohomology $IH^\ast (\M_0)$ with middle perversity is still obscure. In this paper, we   propose a new approach to the cohomology theory  of $\M_0$ for circle actions  which encompasses both the singular and the intersection cohomology of $\M_0$, and sheds new light on  $H^\ast(\M_0;\R)$. 
  
To discuss the background of our approach, recall that one of the main tools for the study of the cohomology of $\M_0$  in the case when $0$ is a regular value of the momentum map is the \emph{regular Kirwan map} from the equivariant cohomology   $H^\ast_G(M)$ of $M$ to $H^{*}(\M_{0};\R)$, which is defined by
\bq\label{eq:28.03.2022}
\begin{tikzcd}
\kappa:H^{*}_{G}(M) \arrow{r}{\iota^{*}} &H^{*}_{G}(\J^{-1}(0)) \arrow{r}{\Car} &H^{*}_{\mathrm{bas} \, G}(\J^{-1}(0)) \arrow{r}{(\pi^{*})^{-1}} &H^{*}(\M_{0};\R),
\end{tikzcd}
\eq
where   $\iota: \J^{-1}(0) \rightarrow M$ and $\pi:\J^{-1}(0) \rightarrow \M_0$ denote the natural injection and projection, respectively,
and $\Car \colon H^{*}_{G}(\J^{-1}(0)) \rightarrow H^{*}_{\mathrm{bas} \, G}(\J^{-1}(0))$ is the Cartan isomorphism between equivariant and basic cohomology of the zero level of the momentum map.
The map $\kappa$ constitutes a surjective ring homomorphism, see \cite{kirwan95}, \cite[Theorem 1]{bott-tolman-weitsman2004}, \cite[Proposition 3.iii)]{JKKW03}, \cite[Section 5.4]{kirwan84}, and by characterizing its kernel, the cohomology of $\M_0$ can be described in terms of the equivariant cohomology of $M$. 
 If $0$ is not a regular value of $\J$,   $\iota^{*}$ still induces a surjection in the Borel model of equivariant cohomology, but a natural surjection from $H^\ast_G(M)$ onto $H^\ast (\M_0;\R)$ is not known so far.  To tackle these difficulties, partial desingularizations of $\M_0$ were introduced  in \cite{kirwan85} in the algebraic case and in \cite[Section 4.1.2]{meinrenken-sjamaar} for Hamiltonian actions in general. To define such desingularizations, one carries out a series of equivariant symplectic blow-ups in order to eliminate points with isotropy in the zero level of the momentum map. In this way one obtains a Hamiltonian manifold $\widetilde{M}$ on which $G$ acts such that $0$ is a regular value of a momentum map $\tilde{\J}$. The symplectic quotient $\widetilde{\M_{0}}:=\widetilde{\J}^{-1}(0)/G$ is called the \emph{partial desingularization of $\M_{0}$} and can be used to define the \emph{intersection Kirwan map} 
\bq
\label{eq:18.08.2023}
\kappa_{IH} \colon H^{*}_{G}(M) \longrightarrow H^{*}_{G}(\widetilde{M}) \longrightarrow H^{*}(\widetilde{\M_{0}}) \longrightarrow IH^{*}(\M_{0}),
\eq
where the first map is induced by the series of blow-ups, the second map is the usual regular Kirwan map and the third map is a certain projection. In the context of  nonsingular, connected, complex, projective varieties and GIT quotients there is a canonical choice of this projection, since $IH^{*}(\M_{0})$ occurs as a summand in $ H^{*}(\widetilde{\M_{0}})$, 
and $\kappa_{IH}$ is surjective, see \cite[Theorem 2.5]{kirwan86}, \cite[Theorem A.1]{jeffrey-mare-woolf2005},  and \cite[p.~220]{JKKW03}.  In the symplectic setting, such a surjection has been established for circle actions in \cite{lerman-tolman00} by using small resolutions of $\M_0$.
But in the general symplectic setting the existence and surjectivity of $\kappa$  remains unknown as is remarked in \cite[p.~223]{kiem-woolf2006}.

 The departing point of our approach is Sjamaar's description \cite{sjamaar05} of the singular cohomology of $\M_0$ by means of differential forms on the stratified space $\M_0$.  To introduce the latter, consider the stratification of $\J^{-1}(0)$ by orbit types and denote the unique open dense  stratum by $\J^{-1}(0)^{\top}$. It is called the  \emph{top stratum}. The corresponding  unique open dense stratum in $\M_0$  given by $\M_0^{\top}:=\J^{-1}(0)^{\top}/G$ is also called the \emph{top stratum}. On the other hand, the \emph{maximally singular stratum} of $\M_0$ is given by the intersection of $\J^{-1}(0)$ with the set of fixed points $M^G$. Its components $F$ constitute submanifolds of possibly different dimensions, and we denote the set of these components by $\F$.  With 
$\iota_\top:\J^{-1}(0)^{\top} \rightarrow M$ and $\pi_\top:\J^{-1}(0)^{\top} \rightarrow \M_0^\top$  
being the natural  inclusion and projection, respectively, one defines the complex of \emph{differential forms on $\M_{0}$} as the subcomplex of differential forms on $\M_0^\top$ given by
\[
\Omega(\M_0):=\left\{ \omega_{0} \in \Omega(\M_0^{\top}) \mid \exists \eta \in \Omega(M) \colon \pi_{\top}^{*}\omega_{0}=\iota_{\top}^{*}\eta \right\}.
\]
The topological significance of this subcomplex stems from the fact that by \cite[Theorem 5.5]{sjamaar05}
\[
H^{*}(\Omega(\M_0),d) \simeq H^{*}(\M_0;\R),
\]
where $d$ is the usual exterior derivative.  

From now on, let $G=S^1\simeq \U(1)$ be the circle group and assume that the $S^1$-action on $\J^{-1}(0)$ is effective. This guarantees that the top stratum is disjoint from the other strata and that $S^1$ acts freely on it. 
We then  extend the complex $\Omega(\M_0)$  as follows. Consider an equivariant symplectic blow-up  $\beta \colon \BLS(M) \rightarrow M$ of $M$ along the set of fixed points $M^ {S^1}$ of the $ {S^1}$-action together with the  diagram
\bq\label{diagram}
\begin{tikzcd}
\BLS(M) \arrow{d}{\beta} &\beta^{-1}\left(\J^{-1}(0)^{\top} \right) \arrow[hook',swap]{l}{\iota'_{\top}} \arrow{d}{\beta_{\top}}  \arrow[bend left=50]{dd}{\pi_{\top}'}\\
M &\J^{-1}(0)^{\top} \arrow[hook']{l}{\iota_{\top}} \arrow{d}{\pi_{\top}}\\
  &\M_{0}^{\top},
\end{tikzcd}
\eq
where $\iota'_\top$ is the natural injection and $\beta_\top$ the natural restriction of $\beta$ to the preimage of $\J^{-1}(0)^\top$. The blow-up $\beta$ is a smooth map from the $ {S^1}$-manifold  $\BLS(M)$ onto $M$, and we define the complex of \emph{resolution differential forms} on $\M_{0}$  as
\[
\widetilde{\Omega}(\M_{0}):=\left\{ \omega_{0} \in \Omega(\M_{0}^{\top}) \mid \exists \tilde{\rho} \in \Omega(\BLS(M)) \colon (\pi_{\top}')^{*}\omega_{0}=(\iota'_{\top})^{*}\tilde{\rho} \right\}
\] 
with the differential $d$. Note that this complex contains $\Omega(\M_0)$ as a subcomplex. It turns out that  this complex is isomorphic to the complex $\Omega(\widetilde \M_{0})$ of differential forms on the  partial desingularization $\widetilde \M_0:=\widetilde{\J^{-1}(0)}/ {S^1}$, where $\widetilde{\J^{-1}(0)}:=\overline{\beta^{-1}(\J^{-1}(0)^{\top})}\subset \BLS(M)$ denotes the strict transform of  $J^{-1}(0)$, compare Proposition \ref{resolutionforms}. Relying on the isomorphism $H^{*}(\widetilde{\Omega}^{*}(\M_{0}),d)\simeq H^{*}(\widetilde\M_{0})$, we obtain in Theorem \ref{longexactsequence} as our first main result a long exact sequence of the form
\[
\ldots \longrightarrow H^{k}\left(\M_{0}\right) \longrightarrow H^{k}(\widetilde{\M}_{0}) \longrightarrow \left( \bigoplus\limits_{{F \in \mathcal{F}_{0}}} {H^{*}(F) \otimes\frac{\mathbb{R}[\widetilde{\sigma}_{0}|_{\widetilde F},\widetilde{\Xi}|_{\widetilde F}]_{\geq 1}}{I_F}} \right)_{k} \longrightarrow H^{k+1}(\widetilde{\M}_{0}) \longrightarrow \ldots,
\]
where  $\mathcal{F}_{0}:=\{ F \subset \J^{-1}(0) \cap M^{S^{1}}\}$ and $[\widetilde{\sigma}_{0}|_{\widetilde F}],[\widetilde \Xi|_{\widetilde F}]\in H^{*}(\widetilde{F})$ denote the restrictions of the symplectic class and the curvature class of $\widetilde{\M}_{0}$ to $\widetilde{F}:=(\beta_{0})^{-1}(F)$, respectively, $\beta_0: \widetilde \M_0 \rightarrow \M_0$ being the projection induced by $\beta$.  In the  special case that $\J^{-1}(0) \cap M^{S^{1}}$ consists only of isolated fixed points and $H^{2k+1}(\M_{0})=0$ for all $k$,  the map $
H^{\ast}(\M_{0})\to H^{\ast}(\widetilde{\M}_{0})
$ is injective so that the above sequence splits and we obtain  a (non-canonical) {isomorphism}
\[
H^{*}(\widetilde{\M}_{0})\cong H^{*}\left(\M_{0}\right) \oplus  \bigoplus\limits_{{F \in \mathcal{F}_{0}}} {\frac{\mathbb{R}[\widetilde{\sigma}_{0}|_{\widetilde F},\widetilde{\Xi}|_{\widetilde F}]_{\geq 1}}{I_F}},
\]
compare Corollary \ref{splitting}. More generally, if for all $k\in \N$ and each component $F \subset{\J^{-1}(0) \cap M^{S^{1}}}$  we have $H^{{2k+1}}(F)=0$, then the natural map
$
H^{2k}(\M_{0})\to H^{2k}(\widetilde{\M}_{0})
$ 
is injective, compare Corollary \ref{generalF}. Consequently,  the image  of $H^{*}(\M_{0};\R)$ in $H^{*}(\widetilde\M_{0})$ under the natural inclusion is an interesting space, even if in general it contains less information than the full singular cohomology $H^{*}(\M_{0};\R)$.  

Now, the major difficulty in defining a surjection  like \eqref{eq:28.03.2022} in the case when  the fixed point set $M^ {S^1}$ of the $ {S^1}$-action is not empty stems from the fact that the Cartan isomorphism involves exterior multiplication with a connection form $\Theta$ on $M$  and its curvature. 
Since any such connection form has  the property  that  $i_{\overline{X}}\Theta=X$, where  $\overline{X} \in \mathcal{X}(M)$ denotes  the fundamental vector field corresponding to an element $X \in \g$, the connection $\Theta$ necessarily must become singular when approaching fixed points of the $ {S^1}$-action, and is therefore only defined on $M \setminus M^ {S^1}$. Consequently, the Cartan map cannot give rise to elements in $H^{*}(\Omega(\M_0),d)$ because their pullbacks  along $\pi^\ast_\top$ must have extensions to the whole of $M$ and,  in particular, to $M^ {S^1}$. In contrast, on  the smooth $ {S^1}$-manifold $\BLS(M)$  the action is locally free on the strict transform $\widetilde{\J^{-1}(0)}$ of $\J^{-1}(0)$ and consequently multiplication with a connection form on $\widetilde{\J^{-1}(0)}$ is defined. In the search of  a generalization of the regular Kirwan map \eqref{eq:28.03.2022} we were therefore led to consider the  \emph{resolution Kirwan map}
\[
\mathcal{K} \colon H^{*}_{ {S^1}}(M) \overset{\beta^{*}}{\longrightarrow} H^{*}_{ {S^1}}(\BLS(M))\overset{\kappa}{\longrightarrow} H^{*}(\widetilde{\M}_{0})
\]
 as the composition of $\beta^{*} \colon H^{*}_{ {S^1}}(M) \rightarrow H^{*}_{ {S^1}}(\BLS(M))$ with the regular Kirwan map of the blow-up $\kappa \colon H^{*}_{ {S^1}}(\BLS(M)) \rightarrow H^{*}(\widetilde{\M_{0}})$.  As our second main result, we then show in Theorem \ref{surjectivity} that the image of  $\mathcal{K}$ contains the image of the natural inclusion $H^{*}(\M_{0};\R)\to  H^{*}(\widetilde\M_{0})$.  This means that the image of $\mathcal{K}$ can be rather large and Theorem \ref{surjectivity} can be regarded as a weak form of Kirwan surjectivity onto $H^{*}(\M_{0};\R)$  in case that $0$ is a singular value of the momentum map. As a special case, if all connected components $F \subset M^{S^{1}}\cap \J^{-1}(0)$ have vanishing odd cohomology, Corollary \ref{generalF} implies that there is  a surjective map 
\bqn
\check \kappa: H^{\mathrm{ev}}_{{S^1}}(M)\to H^{\mathrm{ev}}(\M_{0};\R),
\eqn
compare Corollary \ref{ex:even}; this is the case in the important situation when  $M^{S^1}\cap \J^{-1}(0)$ only consists of isolated fixed points. In case that the odd cohomology of $\M_{0}$ vanishes, we obtain a surjective linear map $ \check \kappa \colon H^{*}_{{S^1}}(M)\to H^{*}(\M_{0};\R)$. Nevertheless, as shown in Corollary \ref{cor:15.11.2023}, for  a  general Hamiltonian $S^1$-manifold $M$, there is no degree-preserving surjection $H^*_{{S^1}}(M)\to H^\ast (\M_{0};\R)$ in odd degrees. Thus, in general, Kirwan surjectivity onto singular cohomology can only hold in even degrees, contrasting with both intersection cohomology and the regular case. Along the same line, as noticed in Remark \ref{rem:18.08.2023}, the resolution Kirwan map $\mathcal{K}$ itself can not be surjective, either. Clearly, $\mathcal{K}$ differs from the map $\kappa_{IH}$ defined in \eqref{eq:18.08.2023} by the projection onto $IH^\ast(\M_0)$.
 
For the proof of Theorem \ref{surjectivity}, we rely again on the concept of resolution differential forms. More precisely, similarly to $\Omega(\M_0)$ and $\widetilde \Omega(\M_0)$ we define  the complex of differential forms  $\Omega^{*}(\J^{-1}(0))$ and the complex of resolution differential forms $\widetilde{\Omega}^{*}(\J^{-1}(0))$ on the zero level of the momentum map. 
Contrasting with $\Omega^{*}(\J^{-1}(0))$, the complex $\widetilde{\Omega}^{*}(\J^{-1}(0))$ is  invariant under multiplication with a connection form, which makes it a $W^\ast$-module and even a  $\mathfrak{g}$-differential graded algebra of type (C) in the sense of \cite[Definition 3.4.1 and Definition 2.3.4]{guillemin-sternberg99}. This is a crucial property because for such a $W^{*}$-module $A$, the map
$
A_{\mathrm{bas} \, \mathfrak{g}} \rightarrow C_{ {S^1}}(A), \, \omega \mapsto 1 \otimes \omega,
$ 
from the basic subcomplex $A_{\mathrm{bas} \, \mathfrak{g}}$  to the Cartan complex $C_ {S^1}(A)$ of $A$ turns out to induce an isomorphism in cohomology with homotopy inverse $\Car \colon C_{ {S^1}}(A) \rightarrow A_{\mathrm{bas} \, \mathfrak{g}}$. Based on this isomorphism we are then able to prove our second main result  in Theorem \ref{surjectivity}. 

As explained in Section \ref{sec:inters}, the complex $\widetilde{\Omega}(\M_{0})$ naturally extends the intersection resolution forms studied in \cite{lerman-tolman00}, which opens up the possibility  to relate $H^\ast(\M_0;\R)$ to   $IH^\ast (\M_0)$ via maps to $H^\ast(\widetilde \M_0)$. As one of the striking features of $H^\ast (\M_0;\R)$, 
note that  the canonical symplectic form $\sigma_0$ on $\M_0^\top$, characterized by the condition that $\pi_{\top}^{*}\sigma_{0}=\iota_{\top}^{*}\sigma$, together with its powers define non-trivial elements in $H^{*}(\Omega(\M_0),d)$, see \cite[Corollary 7.6]{sjamaar05}, but  do not represent cohomology classes in $IH^\ast(\M_0)$. Thus, from a symplectic point of view, the singular cohomology of $\M_0$, as  one of the standard cohomologies for general singular spaces, is more natural than the intersection cohomology of $\M_0$. 

We close our paper by discussing an interesting class of examples given by abelian polygon spaces in Section \ref{sec:17.08.2023}, which have been of interest in symplectic and algebraic geometry as well as in combinatorics for a long time, see \cite{hausmann-knutson1998}.  In particular, our results imply  that there is a linear surjection from equivariant cohomology onto the singular cohomology of abelian polygon spaces in even degrees, see Theorem \ref{thm:27.11.2023}.

In a forthcoming article, and based on the techniques developed in \cite{delarue-ioos-ramacher}, we intend to study the interrelationship of  singular cohomology and intersection cohomology for circle actions more closely within the framework of resolution cohomology via residue formulas. It would also be interesting to generalize our study to non-compact symplectic manifolds under the assumption that the momentum map is proper, see \cite[Remark 5.8]{meinrenken-sjamaar}. In a further study, we also intend to extend our work to general compact group actions. 

{\bf Acknowledgments.}  The first author has received funding from the Deutsche Forschungsgemeinschaft (German Research Foundation, DFG) through the Priority Program (SPP) 2026 ``Geometry at Infinity''. This work is part of the PhD thesis of the third author.

\section{Background and setup}\label{sec:prodcohoms2}
\label{sec:2}

\subsection{Hamiltonian $S^1$-manifolds and their symplectic quotients}\label{sec:sympl.quot} In what follows, we give a summary of the theory of Hamiltonian group actions, restricting ourself  to the case of circle actions, which will be the actions studied in this paper. Thus, let  $G:=S^1$ be the circle group, which we realize as the subgroup of complex numbers of modulus $1$, inducing an identification
\begin{equation}\label{identfla}
\begin{split}
\g\xrightarrow{\,\sim\,}\R, \qquad X \longmapsto x\,,
\end{split}
\end{equation}
in such a way that $X\in\g$ corresponds to $e^{2\pi i x}\in S^1\subset\C$.
This induces in turn an identification $\g^*\simeq\R$ of the dual of the Lie algebra with $\R$ and of the Lebesgue measure on the interval $[0,1]$ with the normalized Haar measure on $S^1$. Further, let $ M $ be a symplectic manifold of dimension $2n$ with symplectic form $\sigma$. Assume that $M$ carries a Hamiltonian action of the circle group ${S^1}$, and
denote the corresponding Kostant-Souriau momentum map  by
\bqn
\J: M \to \g^\ast,  \quad \J(p)(X){=:}J(X)(p),
\eqn
which is characterized by the property 
\bq
dJ(X) =i_{\overline{ X} } \sigma\qquad \forall\;X \in \g,\label{eq:sign24}
\eq
where $\overline{X}$ denotes the fundamental vector field on $M$ associated {with} $X$, $d$ is the de Rham differential, and $i$ denotes contraction.
Note that $\J$ is ${S^1}$-equivariant in the sense that $\J(k^{-1} p) = \Ad^\ast(k) \J(p)$.   Consider the zero  level set $\J^{-1}(\{0\})$, in the following denoted by $\J^{-1}(0)$, and note that fibers of momentum maps are always connected, see e.g.\ \cite[Theorem 1]{atiyah82}. Consider further  the corresponding \emph{symplectic quotient}
\bqn 
\M_0:=\J^{-1}(0)/{S^1}=:M/\!\!/{S^1}. 
\eqn
If $0 \in \g^\ast$ is a  regular value of $\J$, the zero level set is a smooth manifold and $\M_0$ is an orbifold. But in general, they are singular topological spaces stratified by orbit types. More precisely, let $M_{(H)}$ denote the stratum of $M$ of orbit type $(H)$, where $H\subset G$ is a closed subgroup.  One has then the  orbit stratification
$$\J^{-1}(0)=\bigcup_{(H)} \, \, \J^{-1}(0) \cap M_{(H)},  $$
see  \cite{lerman-sjamaar}.  Similarly,  the quotient $\M_0$ is stratified by orbit types according to
\bq
\M_0=  \bigsqcup  \M_{0,(H)}, \qquad \M_{0,(H)}:=(\J^{-1}(0) \cap M_{(H)})/{S^1}.
\eq
On each stratum $\M_{0,(H)}$ there is a natural symplectic form   $\sigma_{(H)}$ characterized by the condition $(\iota_{(H)})^\ast \sigma=(\pi_{(H)})^\ast \sigma_{(H)}$, where $\iota_{(H)}:\J^{-1}(0) \cap M_{(H)}\to M$ is the inclusion map  and $\pi_{(H)}:\J^{-1}(0) \cap M_{(H)} \rightarrow  \M_{0,(H)}$ the orbit map. Both stratifications have a maximal element called the \emph{top stratum},  denoted by   $$\J^{-1}(0)^\top:=\J^{-1}(0) \cap M_{(\mklm{e})}\quad\text{ and }\quad \M_0^\top:=\M_{0,(\mklm{e})},$$ respectively.  Accordingly, we shall write $\iota_\top:=\iota_{(\mklm{e})}$, $\pi_\top:=\pi_{(\mklm{e})}$, and $\sigma_\top:=\sigma_{(\mklm{e})}$.

The top stratum $\M_0^\top$  is {a smooth manifold and dense in $\M_0$}. The so-called \emph{maximally singular stratum} $\M_{0,(S^1)}$ is also a smooth manifold, each of its {connected} components  being identical to some component $F$ of the set of fixed-points
\bqn 
M^{S^1}:=\mklm{p \in M: g \cdot p =p \quad \forall \,  g \in {S^1}}.
\eqn
 The components $F$  are submanifolds {of $M$} of possibly different dimensions, and we denote the set of these components by $\F$.   Furthermore, each $F$ is a symplectic submanifold of $M$,
so that  $TM|_{F}=TF\oplus \Sigma_F$, where $\Sigma_F\subset TM $ is  the symplectic normal bundle of $F$ in $M${, which  carries a symplectic structure}. In particular, the total space of $\Sigma_F$ {is} a symplectic manifold.  Furthermore, ${{S^1}}$ acts on $\Sigma_F$ fiberwise, and we may choose a ${S^1}$-invariant complex structure on $\Sigma_F$ compatible with the symplectic one. Each fiber of the so complexified bundle $\Sigma_F$ then splits into a direct sum of complex $1$-dimensional representations of {${S^1}$}, so that with $ \dim F=2n_F$
\bq
\Sigma_F=\bigoplus_{j=1}^{n-n_F}\E^F_j, \label{eq:splitting111}
\eq
the $\E^F_j$ being complex line bundles over $F$.  The Lie algebra   ${\g}$ acts on them by
\bqn 
(\E_j^F)_p \ni v \mapsto i \lambda^F_j(x) v \in (\E_j^F)_p, \qquad   p \in F, \, x \in {\g}, \, \lambda^F_j\in {\g}^\ast\simeq \R,
\eqn
where $\lambda^F_1,\ldots,\lambda^F_{n-n_F} \in \Z\setminus \{0\}$ are the weights of the ${{S^1}}$-action on $(\Sigma_F)_p$. They do not depend on the point $p\in F$ because $F$ is connected, and they can be grouped into positive weights   $\lambda^F_1,\dots, \lambda^F_{\ell_F^+}$ and negative weights $\lambda^F_{\ell_F^++1},\dots, \lambda^F_{\ell_F^++\ell_F^-}$. The codimension of $F$  in $M$ is given by $\mathrm{codim}\,  F=2(n-n_F)=2(\ell_F^++\ell_F^-)$.  To conclude this subsection, we recall the  local normal form theorem for the momentum map $\J$ due to  Guillemin-Sternberg \cite{guillemin-sternberg84} and Marle
\cite{marle85}, 
which in our situation reduces to the following

\begin{prop}\label{prop:localnormform}
For each component $F \in \F$, there exist 
\begin{enumerate}[leftmargin=*]
\item a faithful unitary representation $\rho_F: S^1 \to (S^1)^{\ell_F^++\ell_F^-} \subset U(\ell_F^+) \times U(\ell_F^-)\subset U(\ell_F^++\ell_F^-)$ with positive weights $\lambda^F_1,\dots, \lambda^F_{\ell_F^+}\in \N$ and negative weights $\lambda^F_{\ell_F^++1},\dots, \lambda^F_{\ell_F^++\ell_F^-}\in -\N$, 
\item a principal $K_F$-bundle $P_F \rightarrow F$, where $K_F$ is a subgroup of $U(\ell_F^+) \times U(\ell_F^-)$ commuting with $\rho_F(S^1)$,
\end{enumerate}
such that 
$$
\Sigma_F\simeq P_F\times_{K_F}\C^{\ell_F^++\ell_F^-},
$$
where $P_F\times_{K_F}\C^{\ell_F^++\ell_F^-}\rightarrow F$ is the vector bundle  associated to $P_F$ by the $K_F$-action. Furthermore, there is a symplectomorphism $\Phi_F:U_F\rightarrow V_F$ from an $S^1$-invariant neighborhood $U_F$ of $F$ in $M$ onto  an $S^1$-invariant neighborhood $V_F$ of the zero section in $\Sigma_F$, which is equivariant  with respect to the $S^1$-action on $\Sigma_F\simeq P_F\times_{K_F}\C^{\ell_F^++\ell_F^-}$ given by $\rho_F$, and 
\bq
\label{eq:30.06.2018}
\J \circ \Phi_F^{-1}([\wp,z])= \frac 12 \sum_{j=1}^{\ell_F^++\ell_F^-} \lambda^F_j |z_j|^2 + \J(F), \quad z =(z_1,\dots,z_{\ell_F^++\ell_F^-}),\; [\wp,z]\in P_F\times_{K_F}\C^{\ell_F^++\ell_F^-},
\eq
where $\J(F)$ is the constant value of  $\J$ on $F$.  In particular,  $2\ell_F^-$ and $2\ell_F^+$ are the dimensions of the negative and positive eigenspaces  of the Hessian of $\J$ at a point of $F$, respectively.
\end{prop}
\begin{proof} See \cite[Lemma 3.1]{lerman-tolman00}. \end{proof}

As mentioned above, the zero level set of the momentum map is connected. Together with  Proposition \ref{prop:localnormform} this implies that if all weights of the isotropy representation of some fixed point component $F$ have the same signs, the {set} $F$ is already the whole zero level set $\J^{-1}(0)$, since $F$ is isolated. Thus from now on, we assume that all fixed point components in the zero level set have positive and negative weights.

\subsection{Singular cohomology of  symplectic quotients}
\label{sec:2.1}

In what follows, we shall review  the de Rham model  developed by Sjamaar in \cite{sjamaar05}  for  the singular or \v Cech cohomology  $H(\M_0;\R)$ with real coefficients  of a symplectic quotient. Quite often, we shall simply write $H(\M_0)$ for $H(\M_0;\R)$. This model will be the departing point for our following considerations and as in the previous subsection we we restrict ourselves to the case where $G:=S^1$, the notation being as before.  Sjamaar's model relies on the notion of a \emph{differential form on $\M_0$}, by which one understands a differential form $\alpha$ on the top stratum $\M_0^\top$  which is induced by a ${S^1}$-invariant differential form $\widetilde \alpha$ on $M$ in the sense that  $$(\pi_\top)^\ast \alpha = (\iota_\top)^\ast \widetilde \alpha.$$
 The space of all differential forms on $\M_0$ is denoted by $\Omega(\M_0)$.  Recall that a form $\beta \in \Omega(M)$ is called \emph{basic} if it is ${S^1}$-invariant and ${S^1}$-horizontal, the latter meaning that $i_{\overline X} \beta=0$ for all $X \in \g$, and notice that if $\tilde \alpha$ induces $\alpha$, then $(\iota_\top)^\ast \widetilde \alpha$ is ${S^1}$-horizontal as a differential form on $\J^{-1}(0)^\top$. It is therefore natural to define $\beta \in \Omega(M)$ to be \emph{$\J$-basic} if it is ${S^1}$-invariant and $(\iota_\top)^\ast \beta$ is ${S^1}$-horizontal on $\J^{-1}(0)^\top$. The set $\Omega_\J(M)$ of $\J$-basic differential forms is a subcomplex of $\Omega(M)$ and the kernel of the surjection $\Omega_\J(M) \rightarrow \Omega(\M_0)$ is the ideal
\bqn 
I_\J(M):= \mklm{ \beta \in \Omega(M)^{S^1} \mid (\iota_\top)^\ast \beta=0 }.
\eqn
Thus, 
\bqn 
\Omega(\M_0) \simeq \Omega_\J(M)/ I_\J(M).
\eqn
Furthermore, there is a well-defined restriction map $\Omega(\M_0) \rightarrow \Omega(\M_{0,(H)})$ for each stratum of $\M_0$, see \cite[Lemma 3.3]{sjamaar05}. Next, define for each open set $U$ of $\M_0$ the set of differential forms $\Omega(U)$ on $U$ in a similar way as the set  $\Omega(\M_0)$, and consider the presheaf of differential graded $\R$-algebras $\Omega:U \longmapsto \Omega(U)$. As can be shown \cite[Lemma 5.1 and 5.4]{sjamaar05}, $\Omega$ constitutes an acyclic sheaf, and the sequence
\bqn 
0 \longrightarrow \underline \R \stackrel{i}{\longrightarrow} \Omega^0 \stackrel{d}{\longrightarrow} \Omega^1 \stackrel{d}{\longrightarrow} \Omega^2 \stackrel{d}{\longrightarrow} \cdots 
\eqn
is exact, where $\underline \R$ is the sheaf of locally constant real-valued functions on $\M_0$ and $i:\underline \R \to \Omega^0$ the natural injection. In other words, the complex $\Omega$ is an acyclic resolution of the constant sheaf, and a standard argument in sheaf theory  \cite[Theorem 5.5]{sjamaar05} implies that  
\bq\label{eq:sjamdR}
H^\ast(\Omega(\M_0),d) \simeq H^\ast(\M_0; \R), \qquad   \text{\emph{(de Rham theorem)}},
\eq
  since $\M_0$ is compact and locally contractible.  
Further, if $\d \M_0^\top:=\sigma_\top^{\dim  \M_0^\top/2}/(\dim  \M_0^\top/2)!$ denotes the volume element of the top stratum it can be shown \cite[Lemma 7.1]{sjamaar05} that for every sufficiently small open set $U$ of $\M_0$ the volume of its top part  is finite. In particular, if $\M_0$ is compact, $\M_0^\top$ has finite volume. Finally, by  \cite[Proposition 7.4]{sjamaar05} one has 
\bq
\label{eq:stokes}
\int_{\M_0^\top} d \beta =0 \qquad   \text{\emph{(Stokes theorem)}}
\eq
for any $\beta \in \Omega^{\dim \M_0^\top-1}(\M_0)$ of compact support. As a consequence, if $\M_0$ is compact, 
\bq
\sigma^k_\top \in H^{2k}(\Omega(\M_0)) \quad \text{ is nonzero for $0 \leq k \leq \dim \M_0^\top/2$},
\eq
see \cite[Corollary 7.6]{sjamaar05}. That is, the symplectic form $\sigma_\top$  and its non-zero powers are  non-trivial  in the singular cohomology of $\M_0$. On the other hand, $H(\M_0;\R)$ may not satisfy Poincar\'e duality. 

\section{Equivariant complex blow-up and partial desingularization of a circle action}

Let $G$ be a Lie group.  For a  proper Hamiltonian action of $G$ on a smooth symplectic manifold $(M,\sigma)$ there is a standard procedure how to desingularize this action, compare  \cite[6.10.~Remark]{kirwan85} and \cite[Section 4]{meinrenken-sjamaar}. This procedure associates with such an action a smooth symplectic manifold $(\BLG(M), \widetilde{\sigma})$, which carries a Hamiltonian $G$-action, where $0$ is a regular value of the momentum map, and a smooth $G$-equivariant map
\[
\beta \colon \BLG(M) \longrightarrow M
\]
such that $\beta \colon \BLG(M) \setminus \beta^{-1}(M^{G}) \rightarrow M \setminus M^{G}$ is an equivariant diffeomorphism.  In this section,  we describe this procedure, which  is based on an iteration of blow-ups,  in the case that $G=S^1$ and explain how it yields a partial desingularization of the symplectic quotient $\M_{0}$.

\subsection{Complex blow-up}\label{sec:blowup}

Let $M$ be a $2n$-dimensional smooth manifold and $N \subset M$  a connected smooth submanifold  of codimension $2k$ whose normal bundle
\[
\nu \colon Q \longrightarrow N, \qquad Q_p:=T_pM/T_pN \text{ for all } p \in M
\]
has a structure group which reduces to $\U(k)$ embedded into $\mathrm{O}(2k)$ as $\mathrm{O}(2k)\cap \mathrm{Sp}(2n,\R)$. Projectivize this bundle fiberwise and denote the resulting bundle by  $\CP(Q) \rightarrow N$. Next, consider the space
\[
L_Q := \{(l,q) \in \CP(Q)\times Q \mid q \in l\},
\]
where $q \in l$ in particular means that $l$ is a line in $Q_{\nu(q)}$, together with  the commutative diagram 
\bq
\begin{tikzcd}
L_Q\arrow{r}{\varphi} \arrow{d}{\lambda} &Q\arrow{d}{\nu}\\
\CP(Q) \arrow{r} &N,
\end{tikzcd}\label{eq:diagram111}
\eq
where the map $\varphi$ sends $(l,q)$ to $q$ and  the map $\lambda$ sends $(l,q)$ to $l$. 

In order to describe the structure of $L_Q$ it is convenient to first introduce  the universal line bundle 
\[
L:=\{(l,z) \in \CP^{k-1} \times \mathbb{C}^{k} \mid z \in l \} \, \longrightarrow \, \CP^{k-1},
\]
whose total space can also be described as   
\begin{align}\label{eq:25.04.2022}
L&=\{(l,z) \in \CP^{k-1} \times \mathbb{C}^{k} \mid l_{i}z_{j}=l_{j}z_{i} \quad \forall\, 1 \leq i < j \leq k \} ,
\end{align}
where $l=[l_1:\dots: l_{k}]$ denote homogeneous coordinates. In particular, $L$ represents the blow-up of $\mathbb{C}^{k}$ along the origin. With this notation, we see that $L_Q$ is a smooth fiber bundle over $N$ with bundle projection $\nu\circ \varphi$ and fiber $L$. Furthermore, $L_Q$ is a smooth line bundle over $\CP(Q)$ with bundle projection given by the map $\lambda$ in \eqref{eq:diagram111}.  

Let further $V \subset Q$ be a closed disc bundle,  diffeomorphic to a closed tubular neighborhood $W \subset M$ of $N$, put $\widetilde{V} := \varphi^{-1}(V)$, and identify $V$ and $W$, which allows us to consider $\varphi|_{\widetilde{V}}$ as a map $\varphi \colon \widetilde{V} \rightarrow  W$. In particular, since $\varphi$ is clearly a diffeomorphism outside the zero section of $Q$, we get that $\varphi|_{\partial\widetilde{V}}:\partial\widetilde{V}\to \partial W$ is a  diffeomorphism. With these preparations, we make the following
\begin{definition}\label{def:blowup}
The \emph{blow-up of $M$ along $N$} is the smooth manifold
\[
\mathrm{Bl}_{N}(M):=\overline{M \setminus W} \cup_{\varphi|_{\partial\widetilde{V}}} \widetilde{V}
\]
obtained by glueing the manifolds with boundary $\overline{M \setminus W}$ and $\widetilde{V}$ with  the gluing map $\varphi|_{\partial\widetilde{V}}:\partial\widetilde{V}\to \partial W$.  The map   $ \beta \colon \mathrm{Bl}_{N}(M) \rightarrow M$ defined by
\[
\beta := \begin{cases} \id \quad \text{on } \overline{M \setminus W}\\
\varphi \quad \text{on } \widetilde{V} \end{cases}
\]
is called the \emph{blow-down map}. The set $\beta^{-1}(N)\subset\widetilde{V}$  is called  the \emph{exceptional divisor} of $\mathrm{Bl}_{N}(M)$. It is the zero section of $L_Q$, regarded as a line bundle over $\CP(Q)$, and can thus be identified with $\CP(Q)$.
\end{definition}
\begin{rem} \label{choices}
The various choices involved in the construction of the blow-up do not cause problems since they lead to equivalent results in the sense that two different choices lead to blow-down maps $\beta \colon \mathrm{Bl}_{N}(M) \rightarrow M$ and $\beta' \colon \mathrm{Bl}_{N}(M)' \rightarrow M$ for which there is a diffeomorphism $f \colon \mathrm{Bl}_{N}(M)\rightarrow \mathrm{Bl}_{N}(M)'$ such that $\beta' \circ f = \beta$. 
\end{rem}
We immediately see that the blow-down map $\beta$ is a diffeomorphism outside the exceptional divisor.

For later use it will be convenient to have  explicit coordinates on $\mathrm{Bl}_{N}(M)$ at  our disposal.  Since $N$ is a codimension $2k$ submanifold of $M$, there is an atlas $\{(U,\varphi_{U})\}$  of $M$ with coordinate maps $
\varphi_{U}(p)=w=(w_{1},\ldots,w_{2n})$ 
such that if $U\cap N \neq \emptyset$, 
\bqn \varphi_{U}(U\cap N)=\{(w_{1},\ldots,w_{2n}) \in \varphi_{U}(U) \mid (w_{1},\ldots,w_{2k})=0 \}.
\eqn
 By shrinking the sets $U$, we can assume without loss of generality that every $U$ with $U\cap N\neq \emptyset$ is contained in the interior of $W$. Recalling the identification $W\simeq V$, the atlas $\{(U,\varphi_{U})\}$ induces an open cover $\mklm{U^\#:=\beta^{-1}(U)}$ of  $\mathrm{Bl}_{N}(M)$ such that 
\begin{itemize}
\item  if $U^\# \cap \beta^{-1}(N) =\emptyset$,  $\beta:U^\#\to U$ is a diffeomorphism;
\item  if $U^\# \cap \beta^{-1}(N) \neq \emptyset$, we have
\[
U^\#\simeq\{ (l,p) \in  \CP^{k-1}\times U   \mid  (l,(z_{1},\ldots,z_{k})) \in L\}, 
\]
where we made the identification $\varphi_{U}(U) \subset \mathbb{R}^{2k} \times \mathbb{R}^{2(n-k)} \simeq \mathbb{C}^{k} \times \mathbb{R}^{2(n-k)}$ and  introduced the complex coordinates $(z_{1},\ldots,z_{k}):=(w_{1},\ldots,w_{2k})=\varphi_U(p)$, where $z_{d}:=w_{d}+i\cdot w_{i+k}$.
\end{itemize}
In the first case, we get coordinates  $\varphi_{U^\#}:=\varphi_{U}\circ \beta$ on $U^\#$ and in the following it will sometimes be convenient to identify $U\simeq U^\#$ and just write $\varphi_{U}$ instead of $\varphi_{U^\#}$. To obtain coordinate maps in the second case, consider the standard open cover $\{V_{i}\}_{1 \leq i \leq k}$ of $\CP^{k-1}$, where $V_{i}:=\mklm{ l \mid l_i \neq 0}$.  Introduce coordinates on $V_{i}$ by setting 
\bq
(u_{1},\ldots,u_{k-1}):=\left(\frac{l_{1}}{l_{i}},\ldots,\frac{l_{i-1}}{l_{i}}, \frac{l_{i+1}}{l_{i}}, \ldots, \frac{l_{k}}{l_{i}} \right).\label{eq:coordui}
\eq
The $V_i$ induces a cover of $U^\#$ by sets $U_i^\#\simeq \mklm{ (l,p) \in  V_i\times U   \mid  (l,(z_{1},\ldots,z_{k})) \in L} $, which by \eqref{eq:25.04.2022}  are given in terms of the equations
\[
z_{j}=z_{i} u_{j}, \quad 1 \leq j<i \quad \text{and} \quad  z_{j}=z_{i} u_{j-1}, \quad i<j\leq k,
\]
yielding the  coordinate maps 
\begin{align*}
\varphi_{U_i^\#} \colon U_{i}^{\#} &\rightarrow \varphi_{U_i^\#}(U_{i}^{\#}),  \qquad  (l,p) \mapsto (u_{1},\ldots,u_{k-1}, z_{i},w_{2k+1},\ldots,w_{2n}).
\end{align*}
 A simple computation then shows that the charts $\{(U^{\#},\varphi_{U^\#})\}$ and $\{(U_i^{\#},\varphi_{U_i^\#})\}$ indeed constitute a smooth atlas for  $\mathrm{Bl}_{N}(M)$, compare \cite[Section 3.1]{igusa00}.

With respect to the charts introduced above, the blow-down map $\beta$ is given near the exceptional divisor as follows:  If $U^\# \cap \beta^{-1}(N) \neq \emptyset$, then $\beta|_{ U^\#}$ amounts to mapping $(l,p)$ to $p$. In the latter case,  if $p \notin N$, one has $z_{i} \neq 0$ for some $1 \leq i \leq k$ and $\beta|_{ U^\#}^{-1}(\{p\})\simeq \{(l,p) \mid l=[z_{1}:\ldots:z_{k}]\}$; on the contrary, if  $p \in N$, then $\beta|_{ U^\#}^{-1}(\{p\})\simeq  \CP^{k-1} \times \{ p \}$. More precisely, in the coordinates provided by $\varphi_{U_i^\#}$ and $\varphi_U$, 
\begin{itemize}
\item  the map $\beta$ is given by the monoidal transformation 
\bqn 
\varphi_U \circ \beta|_{ U_i^\#}\circ \varphi_{U_i^\#}^{-1}:(u_{1},\ldots,u_{k-1}, z_{i},w_{2k+1},\ldots,w_{2n})  \longmapsto (z_i(u_1,\dots, 1, \dots, u_{k-1}),w_{2k+1}, \dots, w_{2n})
\eqn
with $1$ at the $i$-th position,
\item  $\beta^{-1}(N)\cap U_i^\#$ corresponds to the set of points  $\mklm{(l,p) \mid z_i=0}$.
\end{itemize}
In particular, the second statement means that the exceptional divisor $ \beta^{-1}(N) \subset  \mathrm{Bl}_{N}(M) $ is a smooth submanifold of real codimension two.

\subsection{Partial desingularization of a Hamiltonian circle action} \label{pbundleperspective}
The case relevant to us is when $M$ is a  $2n$-dimensional symplectic manifold and  $N \subset M$  is a connected symplectic submanifold of codimension $2k$, which we will assume from now on. In this case, the normal bundle $\nu \colon Q \rightarrow N$ can be identified with the symplectic normal bundle and carries  an almost complex structure. As a consequence,  there is a reduction of the normal frame bundle to $\U(k)$, that is, we get  a principal $\U(k)$-bundle 
$P \rightarrow N$ with
\[
Q \simeq P \times_{\U(k)} \mathbb{C}^{k}.
\]
Thus, the diagram \eqref{eq:diagram111} from the above construction can be written as
\begin{center}
\begin{tikzcd}
 L_Q\simeq P \times_{\U(k)} L \arrow{r}{\varphi} \arrow{d}{\lambda} & P \times_{\U(k)} \mathbb{C}^{k} \simeq Q \arrow{d}{\nu}\\
\CP(Q) \simeq P \times_{\U(k)} \CP^{k-1} \arrow{r} &N
\end{tikzcd}
\end{center}
with the map $\varphi$  given by   $\varphi([\wp,(l,z)])=[\wp,z]$. Compare \cite{mcduff1987} and \cite[Chapter 4]{tralle-oprea97}.

 Assume now  that $M$ carries a Hamiltonian action of $G=S^1$ with momentum map $\J \colon M \to \mathbb{R}$ and recall the notation from Section \ref{sec:sympl.quot}; in particular,  $\F$ denotes the finite set of connected components of $M^{S^1}$. 
 \begin{definition}
We denote by $\BLS(M)$  the \emph{blow-up} of $M$, which we define as the smooth manifold that results from successively blowing up $M$ according to Definition \ref{def:blowup} along all  $F \in \F$, and by $\beta: \BLS(M) \rightarrow M$ the composition of all the blow-down maps.  We call  $E:=\beta^{-1}(M^{S^1})\subset\BLS(M) $ the \emph{exceptional divisor} and for  $F \in \F$ we call $E_F:=\beta^{-1}(F)$  the \emph{exceptional locus} associated with $F$. Furthermore, the \emph{strict transform} is the closure
$$
\widetilde{\J^{-1}(0)}:=\overline{\beta^{-1}(\J^{-1}(0)^{\top})}\subset \BLS(M).
$$
 \end{definition}
Note that the exceptional loci are the connected components of the exceptional divisor.

\begin{prop}\label{prop:desingzero}
The strict transform $\widetilde{\J^{-1}(0)}$ and the exceptional divisor $E$ are smooth submanifolds of $\BLS(M)$ of real codimension $1$, resp. $2$,  with simple normal crossings. More precisely, there is an atlas $\mklm{(\mathscr{U},\varphi_{\mathscr{U}})}$ of $\BLS(M)$ such that for each $\mathscr{U}$ satisfying  $\mathscr{U} \cap \widetilde{\J^{-1}(0)} \cap E_F\neq 0$ for some $F\in \F$ with $\dim F=2k$,   the coordinate map $\varphi_\mathscr{U}(p)=(w_{1},\ldots,w_{2n})\equiv(z_{1},\ldots,z_{k},w_{2n-\dim F+1},\ldots,w_{2n})\in \C^k \times \R^{2n-2k}$ fulfills 
\begin{itemize}
\item $w_{2n-\dim F+1},\ldots,w_{2n}$ are local coordinates on $F$;  
\item  $E_F \cap \mathscr{U}=\{w_i=w_{i+k}=0\}=\{z_{i}=0\}$ for some $1\leq i \leq k$; 
\item $\widetilde{\J^{-1}(0)} \cap \mathscr{U}=\{w_j=0\}$  for some $1\leq j \leq 2n-\dim F$ with $i\neq j$, $i \neq i+k$;
\item $\beta(w_{1},\ldots,w_{2n})=(z_{i}(z_1,\ldots,1, \ldots, z_{k}),w_{2n-\dim F+1},\ldots,w_{2n})$ with $1$ at the $i$-th position, where on the right-hand side we use local coordinates of $M$ in the open set $\beta(\mathscr{U})$.
\end{itemize} 
\end{prop}

\begin{proof} Since the exceptional loci $E_F$ are the connected components of $E$ and at the end of Section \ref{sec:blowup} we saw that each $E_F$ is a smooth codimension $2$-submanifold of $\BLS(M)$,  the same holds for $E$. 

If an  $F \in \F$ satisfies $\widetilde{\J^{-1}(0)} \cap E_F\neq 0$, then $F\subset \J^{-1}(0)$. For such an $F$, recall the description of the symplectic normal bundle $\Sigma_F\simeq P_F\times_{K_F}\C^{\ell_F^++\ell_F^-}$ with $\mathrm{codim}\,  F=2n-\dim F=2(\ell_F^++\ell_F^-)=2k$ and the local normal form of the momentum map near $F$ given in Proposition  \ref{prop:localnormform},  by which the zero level set  is locally described by the relation  
\bqn
\Phi_F(\J^{-1}(0) \cap U_F)= \left\{ [\wp,z_{1},\ldots,z_{\ell_F^++\ell_F^-}] \in V_F \subset  P_F \times_{K_F}\mathbb{C}^{\ell_F^++\ell_F^-} \mid  \sum \lambda^F_{i} |z_i|^2=0  \right\}.
\eqn
Construct an atlas $\mklm{(\mathscr{U},\varphi_{\mathscr{U}})}$ of $\BLS(M)$ following the procedure described at the end of  Section \ref{sec:blowup}, where the atlas $\{(U,\varphi_{U})\}$  of $M$ with $\varphi_U(p)=w=(w_1,\dots,w_{2n})$ underlying this construction  is given in terms of  the local trivializations of the fiber bundle $E_F$ over $F$ and the  symplectomorphisms $\Phi_F:U_F\rightarrow V_F $ with the identification $\C^{\ell_F^++\ell_F^-} \ni z=(z_1,\dots, z_{\ell_F^++\ell_F^-})\equiv(w_1,\dots, w_{2\ell_F^++2\ell_F^-})\in \R^{2\ell_F^++2\ell_F^-}$.  Each $\mathscr{U}$ with  $\mathscr{U} \cap E_F \neq \emptyset$ is then mapped by $\varphi_{\mathscr{U}}$ onto a set of the shape
\[
\mklm{ (l,p) \in  V_i \times U   \mid  (l,(w_{1},\ldots,w_{2\ell_F^++2 \ell_F^-})) \in L}
\]
for some suitable $U$ and $1 \leq i \leq k$, with $w_{2\ell_F^++2 \ell_F^-+1},\ldots,w_{2n}$ local coordinates of $F$ and $\varphi_{\mathscr{U}}(\mathscr{U} \cap E_F)=\mklm{z_i=0}$. Since the group $K_F\subset U(\ell_F^+) \times U(\ell_F^-)$ leaves the quadratic form  $\sum \lambda^F_{i} |z_i|^2$ invariant, we get\footnote{Compare with the proof of \cite[Lemma 3.1]{lerman-tolman00}, as well as \cite[Remark 3.2]{lerman-tolman00} and especially \cite[pp.\ 474--475]{GH94}.}  
 \begin{align*}
\varphi_{\mathscr{U}}\big(\mathscr{U} \cap \beta^{-1}(\J^{-1}(0)^\top)\big) &= \mklm{(l,p)\in V_i \times (\J^{-1}(0)^\top\cap U)   \mid (l,(w_{1},\ldots,w_{2\ell_F^++2 \ell_F^-})) \in L }\\
&=  \bigg \{(l,p)\in V_i \times U   \mid (l,(w_{1},\ldots,w_{2\ell_F^++2 \ell_F^-})) \in L,\\
&\qquad \sum_{k=1}^{\ell_F^++\ell_F^-} \lambda^F_{k} (w_k^2+w_{\ell_F^++\ell_F^-+k}^2)=0,\; (w_{1},\ldots,w_{2\ell_F^++2 \ell_F^-})\neq 0\bigg \}.
  \end{align*}
Thus, 
  \begin{gather}\label{eq:16.04.2022}
\nonumber \varphi_{\mathscr{U}}(\mathscr{U} \cap \widetilde{\J^{-1}(0)})\\ = \bigg \{(l,p)\in V_i \times U \mid ([l_{1}\colon\ldots\colon l_{k}],(w_{1},\ldots,w_{2\ell_F^++2 \ell_F^-})) \in L, \,  \sum_{k=1}^{\ell_F^++\ell_F^-} \lambda^F_{k} |l_{k}|^2=0\bigg \}
  \end{gather}
since the standard coordinates $(u_{1},\ldots,u_{2n-\dim F-1})$ from \eqref{eq:coordui} for the elements $l\in V_i$ reveal that if $(l,(w_{1},\ldots,w_{2\ell_F^++2 \ell_F^-})) \in L$, then
  \bqn 
  \sum_{k=1}^{\ell_F^++\ell_F^-} \lambda^F_{k} (w_k^2+w_{\ell_F^++\ell_F^-+k}^2)=0\iff \sum_{k=1}^{i-1} \lambda^F_{\iota(k)} u_k^2+\lambda^F_i+\sum_{k=i+1}^{2n-\dim F} \lambda^F_{\iota(k)} u_k^2=0
  \eqn
  for suitable indices $\iota(k)$.\footnote{For more details on the computation of  \eqref{eq:16.04.2022} see the following corollary.} Moreover, \eqref{eq:16.04.2022} is actually a local description of $\widetilde{\J^{-1}(0)}$ as a non-singular quadric (recall that $\lambda^F_i\neq 0$), which reveals that $\widetilde{\J^{-1}(0)}$ is a smooth submanifold of $\BLS(M)$. Further, since $\varphi_{\mathscr{U}}(\mathscr{U} \cap E_F)=\mklm{z_i=0}$, $\widetilde{\J^{-1}(0)}$ is transversal to $E_F$. Taking $\sum_{k=1}^{i-1} \lambda^F_{\iota(k)} u_k^2+\lambda^F_i+\sum_{k=i+1}^{2n-\dim F} \lambda^F_{\iota(k)} u_k^2$ as a coordinate and relabling the others proves the third claim. Finally, the statement about the blow-down map in coordinates follows from the description of $\beta$ in coordinates given at the end of  Section \ref{sec:blowup}.
\end{proof}

\begin{cor}
For any component $F$ of the fixed point set the restriction of $\beta$ to the intersection of the exceptional locus $E_{F}$ associated with $F$ and the strict transform $\widetilde{\J^{-1}(0)}$ defines a fiber bundle over $F$ with fiber 
\bq
\left\{l=[l_{1}\colon\ldots\colon l_{k}] \in \CP^{k-1} \, \Big| \, \sum\limits_{i=1}^{k} \lambda_{i}^{F} |l_{i}|^{2}=0 \right\}.\label{eq:fiber}
\eq
\end{cor}

\begin{proof}
Let's compute the strict transform $\widetilde{\J^{-1}(0)}$ more explicitly first. In a neighborhood of $F$ the zero level set is isomorphic to 
\[
\left\{ [p,z] \in V_{F} \subset P_{F} \times_{K_{F}} \mathbb{C}^{k} \, \Big| \, \frac{1}{2}\sum\limits_{i=1}^{k} \lambda_{i}^{F} |z_{i}|^{2}=0 \right\}.
\]
Thus, the inverse image under the blow down map of the non-singular part in this neighborhood is
\[
\left\{ [p,z,l] \in  P_{F} \times_{K_{F}} \mathbb{C}^{k} \times \CP^{k-1} \, \Big| \, [p,z] \in V_{F}, \, \sum\limits_{i=1}^{k} \lambda_{i}^{F} |z_{i}|^{2}=0, \, z \in l, z \neq 0 \right\}.
\]
Using projective coordinates, this is isomorphic to 
\begin{multline*}
\bigg\{ [p,(z_{1},\ldots,z_{k}),[l_{1}:\ldots:l_{k}]] \in  P_{F} \times_{K_{F}} \mathbb{C}^{k} \times \CP^{k-1} \, \Big| \, [p,z] \in V_{F}, \\ \sum\limits_{i=1}^{k} \lambda_{i}^{F} |z_{i}|^{2}=0, \, z_{i}l_{j}=z_{j}l_{i} , z \neq 0 \bigg\}.
\end{multline*}
To proceed, we study this set in a standard open subset $U_{j} :=\{[l_{1}:\ldots:l_{k}] \in \CP^{k} \mid l_{j}\neq 0\} \cong \mathbb{C}^{k-1}$ of complex projective space and obtain for the last set, after renormalizing $l_{j}=1$, the expression
\begin{multline*}
\bigg\{ [p,(z_{1},\ldots,z_{k}),[l_{1}:\ldots :1:\ldots:l_{k}]] \in  P_{F} \times_{K_{F}} \mathbb{C}^{k} \times U_{i}  \ \Big| \, [p,z] \in V_{F},  \sum\limits_{i=1}^{k} \lambda_{i}^{F} |z_{i}|^{2}=0, \\ z_{i}=z_{j}l_{i} , z \neq 0 \bigg\}.
\end{multline*}
Using the relation $z_{i}=z_{j}l_{i}$ this reads
\begin{multline*}
\bigg\{ [p,(z_{j}l_{1},\ldots,z_{j}l_{k},l_{1},\ldots,l_{k})] \in  P_{F} \times_{K_{F}} \mathbb{C}^{k} \times \mathbb{C}^{k-1} \, \Big| \, [p,z] \in V_{F}, \, |z_{j}|^{2}\sum\limits_{i=1}^{k} \lambda_{i}^{F} |l_{i}|^{2}=0,\\ (l_{1},\ldots,z_{j},\ldots,l_{k}) \neq 0 \bigg\}.
\end{multline*}
Since $z_{j}\neq 0$ in this affine piece, we can divide the crucial equation by $|z_{j}|^{2}$, yielding 
\begin{multline*}
\bigg\{ [p,(z_{j}l_{1},\ldots,z_{j}l_{k},l_{1},\ldots,l_{k})] \in  P_{F} \times_{K_{F}} \mathbb{C}^{k} \times \mathbb{C}^{k-1} \, \Big| \, [p,z] \in V_{F}, \, \sum\limits_{i=1}^{k} \lambda_{i}^{F} |l_{i}|^{2}=0, \\(l_{1},\ldots,z_{j},\ldots,l_{k}) \neq 0 \bigg\}.
\end{multline*}
The closure of this is the set
\[
\left\{ [p,(z_{j}l_{1},\ldots,z_{j}l_{k},l_{1},\ldots,l_{k})] \in  P_{F} \times_{K_{F}} \mathbb{C}^{k} \times \mathbb{C}^{k-1} \, \Big| \, [p,z] \in V_{F}, \, \sum\limits_{i=1}^{k} \lambda_{i}^{F} |l_{i}|^{2}=0 \right\}.
\]
In total, we obtain that in a neighborhood of $E_{F}$ the strict transform $\widetilde{\J^{-1}(0)}$ is diffeomorphic to
\[
\left\{ [p,z,l] \in  P_{F} \times_{K_{F}} \mathbb{C}^{k} \times \CP^{k-1} \, \Big| \, [p,z] \in V_{F}, \, \sum\limits_{i=1}^{k} \lambda_{i}^{F} |l_{i}|^{2}=0, \, z \in l \right\}
\]
which proves our claim.
\end{proof}

\begin{rem} \label{ProjectiveZeroSetRemark}
Notice that the fiber \eqref{eq:fiber} is the zero level set $\J^{-1}_{\CP^{k},\lambda}(0)$ of the standard Hamiltonian circle action
\[
S^{1}\times \CP^{k-1} \rightarrow \CP^{k-1},\qquad z \cdot [l_{1}:\ldots:l_{k}]:=[z^{\lambda_{1}}l_{1}:\ldots:z^{\lambda_{k}}l_{k}]
\]
with momentum map 
\begin{align*}
\J_{\CP^{k},\lambda} \colon \CP^{k-1} &\longrightarrow \mathbb{R}, \qquad [l_{1}:\ldots:l_{k}] \longmapsto \frac{\sum\limits_{i=1}^{k} \lambda_{i} \cdot |l_{i}|^{2}}{\sum\limits_{i=1}^{k} |l_{i}|^{2}}
\end{align*}
and $0$ is a regular value in this case, see \cite[Section 5.1]{kalkman95}. As a consequence, $S^{1}$ acts locally freely on $\J^{-1}_{\CP^{k},\lambda}(0)$. We will denote this complex projective space with the Hamiltonian circle action induced by the isotropy representation of $F$ by $\CP_{\lambda, F}^{k-1}$. 
\end{rem}

   The  \emph{partial resolution or desingularization}  of the ${S^1}$-action on $M$ is the content of  the following 
\begin{prop}
$\BLS(M)$ carries a unique ${S^1}$-action making $\beta : \BLS(M) \rightarrow M$ equivariant. Furthermore, when restricted to $\widetilde{\J^{-1}(0)}$, the ${S^1}$-action is locally free. 
\end{prop}
\begin{proof} 
Let $F \in \mathcal{F}$  and recall the description of the symplectic normal bundle $\Sigma_F\simeq P_F\times_{K_F}\C^{\ell_F^++\ell_F^-}$ of $F$ given in Proposition  \ref{prop:localnormform}. Blowing up $M$ along  $F$  leads to  the diagram
\begin{center}
\begin{tikzcd}
 P_F \times_{K_F} L \arrow{r}{\varphi} \arrow{d} & P_F \times_{K_F} \mathbb{C}^{\ell_F^++\ell_F^-}\simeq \Sigma_F\arrow{d}{\nu}\\
 P_F \times_{K_F} \CP^{\ell_F^++\ell_F^--1} \arrow{r} &F
\end{tikzcd}
\end{center}
where $L$ is defined using $k=\ell_F^++\ell_F^-$ in the notation above.  
The sphere $S^{2\ell_F^++2\ell_F^--1}\subset \mathbb{C}^{\ell_F^++\ell_F^-}$ is invariant under the $S^{1}$-action on $\mathbb{C}^{\ell_F^++\ell_F^-}$ from Proposition  \ref{prop:localnormform}. Since the latter also commutes with the diagonal action on the sphere defining projective space, it induces an $S^1$-action on  $\CP^{\ell_F^++\ell_F^--1}$. Thus, $S^1$ acts on the line bundle $L$ by means of the product action $S^{1} \curvearrowright \CP^{\ell_F^++\ell_F^--1} \times \mathbb{C}^{\ell_F^++\ell_F^-}$, and since this action commutes with the $K_F$-action by Proposition  \ref{prop:localnormform}, we obtain an $S^1$-action on $L_Q\simeq P_F \times_{K_F} L$. The map $\varphi \colon P \times_{K_F} L \rightarrow P_F \times_{K_F} \mathbb{C}^{\ell_F^++\ell_F^-}$ is $S^1$-equivariant, which means that the $S^1$-actions on $L_Q$ and $M$ glue together to an $S^1$-action on $\BLS(M)$.

This action is locally free away from the exceptional divisor, since it coincides with the original ${S^1}$-action on $M$ there, while on the exceptional divisor the action equals the isotropy action on the complex projective space. When we restrict this action to the strict transform $\widetilde{\J^{-1}(0)}$ it is locally free by Remark \ref{ProjectiveZeroSetRemark}.

The equivariance of $\beta$ is immediate from the construction. Finally, the uniqueness of the action follows from the fact that the complement of the exceptional divisor, on which the action is uniquely determined by the ${S^1}$-action on $M$, is dense in $\BLS(M)$. 
\end{proof} 

As a consequence we obtain

\begin{cor}
There is a connection form $\Theta \in \Omega^{1}(\J^{-1}(0)^{\top})$ for the ${S^1}$-action on the regular part of $\J^{-1}(0)$ such that there is a connection form $\widetilde{\Theta} \in \Omega^{1}(\widetilde{\J^{-1}(0)})$  for the ${S^1}$-action on $\widetilde{\J^{-1}(0)}$ which extends  $\Theta$ in the sense that   $\widetilde{\Theta}\big|_{\beta^{-1}(\J^{-1}(0)^{\top})}=\beta_{\top}^{*}  \Theta$. 
\end{cor}

\begin{rem}
Note that one could have hoped that the action of $S^{1}$ on the whole space $\BLS(M)$ was (locally) free, but in order for the $S^{1}$-action to be (locally) free on the blow-up, it has in particular to be (locally) free on the exceptional divisor. But for a compact manifold $X$ acted on by a torus $T$, the Euler characteristic satisfies $\chi(X)=\chi(X^{T})$ \cite[Theorem 9.3]{goertscheszoller}. The Euler characteristic of $\mathbb{C}P^{k}$ is $k+1$, so there cannot exist a (locally) free $S^{1}$-action on $\mathbb{C}P^{k}$, since in that case  the fixed point set would be empty and $\chi(\emptyset)=0\neq k+1$.
\end{rem}

\begin{definition}
The $\emph{partial desingularization}$ of $\M_{0}$ is the orbifold $\widetilde{\M}_{0}:=\widetilde{\J^{-1}(0)}/{S^1}$ and comes together with the continuous map $\beta_0$ defined by the diagram

\begin{center}
\begin{tikzcd}
\widetilde{\M}_{0} \arrow{r}{\beta_{0}} &\M_{0}\\
\widetilde{\J^{-1}(0)} \arrow{r}{\beta} \arrow{u}{\widetilde{\pi}} &\J^{-1}(0) \arrow{u}{\pi}.
\end{tikzcd}
\end{center}
We denote the exceptional fiber bundles of $\beta_{0}$ by
\[
\beta_{0}^{F}\colon \widetilde{F}:=(\beta_{0})^{-1}(F) \longrightarrow F.
\]
According to Remark \ref{ProjectiveZeroSetRemark} they are of the form

\[
\CP_{\lambda, F}^{k-1}/\!\! /S^{1} \rightarrow \widetilde{F} \rightarrow F,
\]
where $2k$ is the codimension of $F$ in $M$.
\end{definition}
The curvature form $\widetilde \Xi:=d \widetilde{\Theta} \in\Omega^{2}(\widetilde{\J^{-1}(0)})$ is basic and therefore descends to a form $\widetilde\Xi\in \Omega^{2}(\widetilde{\M}_{0})$ still called curvature form and denoted by the same letter.
\begin{rem}
It is  worthwile to think of $\beta_{0} \colon \widetilde{\M}_{0} \rightarrow \M_{0}$ as a desingularization and compare it to other desingularizations of the singular quotient $\M_{0}$. Lerman-Tolman's desingularization \cite{lerman-tolman00} is rather different from ours, since they construct a small resolution of $\M_{0}$ by perturbing the momentum map. More precisely,  a resolution $h$ of a simple stratified space $X$ is defined as a surjective continuous map $h \colon \widetilde{X} \rightarrow X$, where $\widetilde{X}$ is an orbifold, such that $h^{-1}(X^{\top})$ is dense in $\widetilde{X}$ and $h|_{h^{-1}(X^{\top})}$ is a diffeomorphism, where $X^{\top}$ is the open, dense top stratum. It is called {\emph{small}} if
\[
\codim\left( \left\{ x \in X \mid \dim(h^{-1}(x))\geq r \right\} \right) > 2r.
\]
Let us consider $\beta_{0} \colon \widetilde\M_{0} \rightarrow \M_{0}$ and some $x \in F \in \mathcal{F}$. Then on the one hand
\[
\codim_{\M_{0}}(F)=2n-2n_{F}-2, 
\] 
while on the other hand
\[
\dim((\beta_{0})^{-1}(x))=\dim\left( \CP^{n-n_{f}-1}/\!\!/S^{1} \right) = 2n-2n_{F}-4
\]
and even though $\widetilde{\beta}$ is a resolution of $\M_{0}$, it is in general not small, our exceptional fibers are ``too big''.
 A similar dimension count shows that the shift-desingularization obtained by reducing $M$ at a regular value near $0$ is in general not small as well, see \cite[Section 2.3]{guillemin94}. The connection between our partial desingularization and the shift desingularization is provided by the next section and \cite[Section 4.3]{meinrenken-sjamaar}.
\end{rem}
To close this subsection, let us still give another characterization of the space $\CP_{\lambda, F}^{k-1}/\!\! /S^{1}$. Let $S^{1}$ act on $\mathbb{C}^{k}$ as induced by the isotropy representation of a fixed point component $F$. We can decompose $  \mathbb{C}^{k}=\mathbb{C}^{k_{+}} \times \mathbb{C}^{k_{-}}$ into positive and negative weight spaces and define
\[
S_{F}^{\pm}:=\left\{(z_{1},\ldots, z_{k_{\pm}} ) \in \mathbb{C}^{k_{\pm}} \mid \sum\limits_{i=1}^{k_{\pm}} \pm \lambda^{F}_{i} |z_{i}|^{2}=1 \right\}.
\]
Now the isotropy representation induces an $S^{1}$ action on these ellipsoids by 
\[
z \cdot(z_{1},\ldots, z_{k_{\pm}}):= (z^{\lambda_{1}}\cdot z_{1},\ldots, z^{\lambda_{k_{\pm}}} \cdot z_{k_{\pm}}).
\]

\begin{prop}\label{bisphere}
If all weights of the $S^1$-action on $\mathbb{C}^{k}$  have the same absolute value, there is a diffeomorphism
\[
\CP_{\lambda, F}^{k-1}/\!\! /S^{1} \longrightarrow (S_{F}^{+}/S^{1}) \times (S_{F}^{-}/S^{1}).
\]
\end{prop}

\begin{proof}
Let $[l_{1}: \ldots:l_{k}] \in \CP^{k-1}$ be such that $\sum\limits_{i=1}^{k}\lambda_{i}^{F} |l_{i}|^{2}=0$. Define 
\[
Q:=\sum\limits_{i=1}^{k_{+}}\lambda_{i}^{F} |l_{i}|^{2}=\sum\limits_{i=1}^{k_{-}}-\lambda_{k_{+}+i}^{F} |l_{k_{+}+i}|^{2}.
\] 
Then the {smooth} map
\begin{align*}
\F \colon \CP_{\lambda, F}^{k-1}/\!\! /S^{1} &\longrightarrow (S_{F}^{+}/S^{1}) \times (S_{F}^{-}/S^{1})\\
{S^1\cdot}[l_{1}: \ldots:l_{k}] &\longmapsto \left({S^1\cdot} \frac{1}{\sqrt{Q}}(l_{1},\ldots,l_{k_{+}}),{S^1\cdot}\frac{1}{\sqrt{Q}}(l_{k_{+}+1},\ldots,l_{k_{+}+k_{-}}) \right)
\end{align*}
is well-defined with {smooth} inverse
\begin{align*}
\G: (S_{F}^{+}/S^{1}) \times (S_{F}^{-}/S^{1}) &\longrightarrow \CP_{\lambda, F}^{k-1}/\!\! /S^{1}  \\
\left( {S^1\cdot}(l_{1},\ldots,l_{k_{+}}),{S^1\cdot}(l_{k_{+}+1},\ldots,l_{k_{+}+k_{-}}) \right) &\longmapsto  {S^1\cdot}[l_{1}: \ldots:l_{k}]. 
\end{align*} 
Indeed, {denote by $\tilde \F$ the map of which we claim that it induces $\F$ by passing to the $S^1$-quotient in the  domain, that is, $\tilde \F$ is defined like $\F$ but without the  ``$S^1\cdot$'' on the left-hand side.}  Consider first $z \in S^{1}$. Then {$\tilde \F$ } is given by 
\begin{align*}
&[z^{\lambda_{1}} \cdot l_{1}: \ldots: z^{\lambda_{k}}  \cdot l_{k}] \\
&\qquad \mapsto \left({S^1\cdot} \frac{1}{\sqrt{Q}}(z^{\lambda_{1}}  \cdot l_{1},\ldots,z^{\lambda_{k_{+}}} \cdot l_{k_{+}}),{S^1\cdot}\frac{1}{\sqrt{Q}}(z^{\lambda_{k_{+}+1}} \cdot l_{k_{+}+1},\ldots, z^{\lambda_{k_{+}+k_{l}}}  \cdot l_{k_{+}+k_{-}}) \right)\\
 &\qquad=\left( {S^1\cdot}\frac{1}{\sqrt{Q}}(l_{1},\ldots,l_{k_{+}}),{S^1\cdot}\frac{1}{\sqrt{Q}}(l_{k_{+}+1},\ldots,l_{k_{+}+k_{-}}) \right).
\end{align*}
Secondly, consider $w \in \mathbb{C} \setminus \{0\}$. Then
\[
\sum\limits_{i=1}^{k_{+}}\lambda_{i}^{F} |w \cdot l_{i}|^{2}=|w|^{2} \cdot Q
\]
and $\tilde \F$ is given by 
\begin{align*}
&[w \cdot l_{1}: \ldots: w  \cdot l_{k}] \\
&\qquad\mapsto \left({S^1\cdot} \frac{1}{|w| \cdot \sqrt{Q}}(w  \cdot l_{1},\ldots,w\cdot l_{k_{+}}),{S^1\cdot}\frac{1}{|w| \cdot \sqrt{Q}}(w\cdot l_{k_{+}+1},\ldots, w \cdot l_{k_{+}+k_{-}}) \right)\\
&\qquad=\left({S^1\cdot} \frac{1}{\sqrt{Q}}(\frac{w}{|w|}  \cdot l_{1},\ldots,\frac{w}{|w|} \cdot l_{k_{+}}),{S^1\cdot}\frac{1}{\sqrt{Q}}(\frac{w}{|w|} \cdot l_{k_{+}+1},\ldots, \frac{w}{|w|}  \cdot l_{k_{+}+k_{-}}) \right).
\end{align*}
Now, there are $z_{1}, z_{2} \in S^{1}$ such that $z_{1}^{\lambda}=z_{2}^{-\lambda}=\frac{w}{|w|}$, where $\lambda^{F}_{1}=\ldots=\lambda^{F}_{k_{+}}=:\lambda$ and $\lambda^{F}_{k_{+}+1}=\ldots=\lambda^{F}_{k_{+}+k_{-}}=-\lambda$ by assumption, so that
\begin{align*}
&\left( {S^1\cdot}\frac{1}{\sqrt{Q}}(\frac{w}{|w|}  \cdot l_{1},\ldots,\frac{w}{|w|} \cdot l_{k_{+}}),{S^1\cdot}\frac{1}{\sqrt{Q}}(\frac{w}{|w|} \cdot l_{k_{+}+1},\ldots, \frac{w}{|w|}  \cdot l_{k}) \right)\\
&= \left({S^1\cdot} \frac{1}{\sqrt{Q}}(z_{1}^{\lambda}  \cdot l_{1},\ldots,z_{1}^{\lambda} \cdot l_{k_{+}}),{S^1\cdot}\frac{1}{\sqrt{Q}}(z_{2}^{-\lambda} \cdot l_{k_{+}+1},\ldots, z_{2}^{-\lambda}  \cdot l_{k}) \right)\\
&=\left( {S^1\cdot}\frac{1}{\sqrt{Q}}( l_{1},\ldots, l_{k_{+}}),{S^1\cdot}\frac{1}{\sqrt{Q}}( l_{k_{+}+1},\ldots,  l_{k}) \right).
\end{align*}
Thus, $\F$ is well-defined. On the other hand, take $z_{1},z_{2} \in S^{1}$ and let $z_{1}',z_{2}' \in S^{1}$ be such that $(z_{i}')^{2}=z_{i}$. Then
\begin{align*}
[z_{1}^{\lambda} \cdot l_{1}:\ldots :z_{2}^{-\lambda} \cdot l_{k}]&=[(z_{1}'z_{2}')^{\lambda}\cdot \frac{(z_{1}')^{\lambda}}{(z_{2}')^{\lambda}} \cdot l_{1}:\ldots:(z_{1}'z_{2}')^{-\lambda}\cdot \frac{(z_{1}')^{\lambda}}{(z_{2}')^{\lambda}} \cdot l_{k}]\\
&=[(z_{1}'z_{2}')^{\lambda} \cdot l_{1}:\ldots:(z_{1}'z_{2}')^{-\lambda}\cdot l_{k}]\\
&={(z_{1}'z_{2}')\cdot}[ l_{1}:\ldots: l_{k}]
\end{align*}
implying that $\G$ is well-defined.
\end{proof}

\subsection{Symplectic blow-ups, symplectic cuts and partial desingularization}

Based on ideas of Gromov, McDuff \cite{mcduff1987} showed that the blow-up construction described above is compatible with the symplectic structure on $M$. In fact, one has the following 
\begin{prop}
There is a symplectic form $\widetilde{\sigma} \in \Omega^{2}(\BLS(M))$ such that
\begin{enumerate}[label=(\roman*)]
\item $\widetilde{\sigma}=\beta^{*}\sigma$ outside a neighborhood of all exceptional loci $E_{F}$;
\item when restricted to a fiber of the exceptional loci $E_{F} \rightarrow F$, $\widetilde{\sigma}$ equals $\varepsilon\cdot \sigma_{\mathrm{FS}}$ for some $\eps>0$, where $\sigma_{\mathrm{FS}} \in \Omega^{2}(\CP^{k-1})$ is the Fubini--Study form.
\end{enumerate}
\end{prop}
This was used by Kirwan \cite[Remark 6.10]{kirwan85} in the setting of general compact group actions to give a first idea of a partial desingularization of a symplectic quotient by a succession of symplectic blow-ups along certain isotropy strata followed by a symplectic reduction. It was made rigorous by Meinrenken-Sjamaar \cite[Section 4]{meinrenken-sjamaar} by means of the symplectic cut technique of Lerman \cite{lerman95}, which we recall briefly for circle actions.  Namely, suppose that $S^{1}$ acts on a symplectic manifold $(M,\sigma)$ in a Hamiltonian fashion with momentum map
\[
\J \colon M \rightarrow \mathbb{R}.
\]
Now, consider the $S^{1}$-action on the product $(M \times \mathbb{C}, {\sigma} \oplus dx \wedge dy)$ defined by $z \cdot (p,z'):=(z \cdot p,z^{-1}\cdot z')$. This is again a Hamiltonian action with momentum map
\begin{align*}
\overline{\J} \colon M \times \mathbb{C} &\rightarrow \mathbb{R}, \qquad 
(p,z') \mapsto \J(p)-\frac{1}{2}|z'|^{2}.
\end{align*}
Then if $\varepsilon$ is a regular value of $\J$ it is also a regular value of $\overline{\J}$ and 
\begin{align*}
\overline{\J}^{-1}(\varepsilon)&=\left\{ (p,z') \mid \J(p)>\varepsilon, {|z'|= }\sqrt{2(\J(p)-\varepsilon)} \right\}\, \cup \, \left\{(p,0) \mid \J(p)=\varepsilon \right\}\\
&\cong \left(\J^{-1}(({\eps} , \infty))\times S^{1} \right) \, \cup \, \J^{-1}(\varepsilon). 
\end{align*}
The {\emph{symplectic cut}} of $(M,{\sigma})$ at $\varepsilon$ is defined as the symplectic reduction
\[
(\overline{M}_{\J \geq \varepsilon},\overline{{\sigma}}_{\varepsilon}):=(M \times \mathbb{C})/\!\!/_{\varepsilon}S^{1}:=\overline{\J}^{-1}(\varepsilon)/S^{1} \cong \J^{-1}((\varepsilon,\infty)) \, \cup \, M/\!\!/_{\varepsilon}S^{1}.
\]
If furthermore another Lie group $G$ acts on $(M,{\sigma})$ with momentum map $\J_{G}\colon M \rightarrow \mathfrak{g}^{*}$ and the $G$-action commutes with the $S^{1}$-action, then $G$ also acts on $(\overline{M}_{\J \geq \varepsilon},\overline{{\sigma}}_{\varepsilon})$ in a Hamiltonian way with momentum map
\begin{align} \label{momentmapblowup} 
\overline{\J_{G}} \colon \overline{M}_{\J \geq \varepsilon} &\rightarrow \mathfrak{g}^{*}, \qquad 
p \mapsto \begin{cases}\J_{G}(p), \text{ if } \, p \in \J^{-1}((\varepsilon,\infty)),\\
\J_{G}(q), \text{ if } \, p= [q] \in M/\!\!/_{\varepsilon}S^{1}.
 \end{cases}
\end{align}

As one easily sees, the symplectic blow up along a symplectic submanifold is a special case of the symplectic cut where one considers a tubular neighborhood of the symplectic submanifold and the $S^{1}$-action of fiberwise scalar multiplication as explained in \cite[Section 4.1.1]{meinrenken-sjamaar}. Meinrenken--Sjamaar define the partial desingularization of $\M_{0}$ as the  regular symplectic quotient $\BLS(M)/\!\!/S^{1}$ of $\BLS(M)$ at $0$, see \cite[Section 4.1.2]{meinrenken-sjamaar}. From the above expression \eqref{momentmapblowup} of the {momentum} map of the blow-up with $G=S^1$ and Proposition \ref{prop:localnormform} one then obtains

\begin{cor}\label{TransformZeroSet}
Our partial desingularization $\widetilde{\M}_{0}=\widetilde{\J^{-1}(0)}/{S^1}$ is diffeomorphic to Kirwan--Meinren\-ken--Sjamaars's partial desingularization $\BLS(M)/\!\!/S^{1}$. 
\end{cor}

With the symplectic form of $\BLS(M)$ and the notation of Remark \ref{ProjectiveZeroSetRemark} at hand, we can understand the topology of the exceptional fiber bundles 

\[
\CP_{\lambda, F}^{k-1}/\!\! /S^{1} \rightarrow \widetilde{F} \rightarrow F
\]
more deeply. In fact, one has the following 

\begin{prop}\label{exceptionalcohomology} 
Let $[\widetilde{\sigma}_{0}|_{\widetilde F}],[\widetilde \Xi|_{\widetilde F}]\in H^{*}(\widetilde{F})$ be the restrictions of the symplectic class and the curvature class of $\widetilde{\M}_{0}$ to $\widetilde{F}$, respectively. Then
\[
H^{*}(\widetilde{F}) \cong H^{*}(F) \otimes \frac{\mathbb{R}[{\widetilde{\sigma}_{0}|_{\widetilde F},\widetilde \Xi|_{\widetilde F}}]}{{I_F}}, 
\]
where {$I_F$} is an ideal encoding relations between $[{\widetilde{\sigma}_{0}|_{\widetilde F}}]$ and $[{\widetilde \Xi|_{\widetilde F}}]$.
\end{prop}

\begin{proof}
By a theorem of Kalkman \cite[Theorem 5.2. and Remarks (5.3)]{kalkman95}, the cohomology of the fiber is isomorphic to
\[
H^{*}(\CP_{\lambda, F}^{k-1}/\!/\!S^{1})\cong \frac{\mathbb{R}[\phi,\eta]}{I},
\] 
where $\phi$ {and $\eta$ are the cohomology classes represented by the symplectic form of $\CP_{\lambda, F}^{k-1}/\!\!/S^{1}$  and the curvature class of the $S^{1}$-bundle $\J^{-1}_{\lambda,F}(0) \rightarrow \CP_{\lambda, F}^{k-1}/\!\!/S^{1}$, respectively}. Moreover, $I$ is an ideal of relations between these generators which is specified in \cite{kalkman95}.
Since both of the generators are restrictions of classes on $\widetilde{F}$, namely the restriction of the symplectic class $[\widetilde{\sigma}_{0}]$ and the curvature class $[{\widetilde \Xi}]$ of {$\widetilde{\M}_{0}$} to $\widetilde{F}$, we may apply Leray--Hirsch's theorem, see \cite[Theorem 5.11]{bott-tu}, which finishes the proof.
\end{proof}
Let $S^{1}$ act on $\mathbb{C}^{k}$ as induced by the isotropy representation of a fixed point component $F$. We can define connection forms on $S_{F}^{\pm}$ by
\[
\Theta_{\pm}:= \pm \sum\limits_{i=1}^{k_{\pm}}x_{i}dy_{i}-y_{i}dx_{i}.
\]
The basic forms $d\Theta_{\pm}$ descend to forms on $S_{F}^{\pm}/S^{1}$ denoted by the same name.
\begin{prop}\label{pullbackbisphere}
 If all weights of the {$S^{1}$-action on $\mathbb{C}^{k}$} have the same absolute value,  the diffeomorphism
\[
\CP_{\lambda, F}^{k-1}/\!\! /S^{1} \longrightarrow (S_{F}^{+}/S^{1}) \times (S_{F}^{-}/S^{1}),
\]
{which exists by Proposition \ref{bisphere}}, pulls back the form $\frac{1}{2}(d\Theta_{+}-d\Theta_{-})$ to the symplectic {form} and $d\Theta_{+}+d\Theta_{-}$ to the curvature {form} of $\CP_{\lambda, F}^{k-1}/\!\! /S^{1}$.
\end{prop}
\begin{proof}
The symplectic {form} of $\CP_{\lambda, F}^{k-1}/\!\! /S^{1}$  is induced by restricting the standard symplectic form $\sum\limits_{i=1}^{k} dx_{i} \wedge dy_{i}$ to $S^{2k-1}$, so it equals $\frac{1}{2}(d\Theta_{+}-d\Theta_{-})$. The same reasoning applies to the curvature {form}.
\end{proof}

\section{Resolution differential forms and their underlying $\mathfrak{g}$-differential graded algebra}\label{sec:resdiff}

Let the notation be as in previous sections. In this section we begin our study of the cohomology of $\M_0$ for circle actions by giving a new description of  $H(\widetilde{\M}_{0})$  departing from Sjamaar's de Rham model for $H(\M_0;\R)$. This alternative description will allow us to relate $H(\widetilde{\M}_{0})$ to $H(\M_0;\R)$ and will be crucial in the study of the surjectivity of the resolution Kirwan map in the next section.    We use the notation from the previous sections. 

\begin{definition}
With the notation of Diagram \eqref{diagram} we define the \emph{cochain complex of resolution forms} on $\M_{0}$ as
\[
\widetilde{\Omega}(\M_{0}):=\left\{ \omega_{0} \in \Omega(\M_{0}^{\top}) \mid \exists\, \widetilde{\rho} \in \Omega(\BLS(M)) \colon (\pi_{\top}')^{*}\omega_{0}=(\iota'_{\top})^{*}\widetilde{\rho} \right\},
\]
with the exterior derivative $d$ as the differential.  
\end{definition}

The first thing to notice about this definition is that it does not depend on the choices involved in the definition of the blow-up $\BLS(M)$ by Remark \ref{choices}. Another useful observation is that, even though for a given form $\omega_{0} \in \Omega(\M_{0}^{\top})$ the existence of a form $\widetilde{\rho} \in \Omega(\BLS(M))$ satisfying
\[
\beta_{\top}^{*}\pi_{\top}^{*}\omega_{0}=(\pi_{\top}')^{*}\omega_{0}=(\iota'_{\top})^{*}\widetilde{\rho}
\] 
might seem quite restrictive because the exceptional divisor $E\subset\BLS(M)$ is a ``larger space'' than the fixed point set $M^{S^1}$, it is actually less restrictive than being a differential form on $\M_{0}$  because for any $p\in\beta^{-1}(\J^{-1}(0)^{\top})$  one has
\[
(\beta_{\top}^{*}\pi_{\top}^{*}\omega_{0})_{p}(v_{1},\ldots,(\beta_{\top})_{*}v_{\mathrm{deg}\,\omega_0})=(\pi_{\top}^{*}\omega_{0})_{\beta(p)}((\beta_{\top})_{*}v_{1},\ldots,(\beta_{\top})_{*}v_{\mathrm{deg}\,\omega_0}),
\]
and the kernel of $\beta_{*}$ at any point $p_0\in E$ contains the tangent space $T(E_F)_{\beta(p_0)}$ of the fiber of $E_F$ over $\beta(p_0)$, where $F\in \F$ is such that  $p_0\in E_F$. This makes extending $\beta_{\top}^{*}\pi_{\top}^{*}\omega_{0}$ to $E_{F}$ potentially easier than extending $\pi_{\top}^{*}\omega_{0}$ to $F$ -- informally, one could describe this as the phenomenon that there is ``more room'' for finding an extension in $\BLS(M)$ than there is in $M$.

Now, let
\[
{\Omega}_{\J}(\BLS(M)):=\left\{ \widetilde \omega \in \Omega(\BLS(M))^{{S^1}} \mid \widetilde \omega|_{\beta^{-1}\left( \J^{-1}(0)^{\top} \right)} \text{ is ${S^1}$-horizontal} \right\}
\]
and
\[
I_{\J}(\BLS(M)):=\left\{ \widetilde \omega \in \Omega(\BLS(M))^{{S^1}} \mid \widetilde \omega|_{\beta^{-1}\left( \J^{-1}(0)^{\top} \right)}=0 \right\}.
\]
\begin{prop}\label{propquotient}
We have an isomorphism of cochain complexes
\[
\widetilde{\Omega}(\M_{0}) \simeq \frac{{\Omega}_{\J}(\BLS(M))}{I_{\J}(\BLS(M))}.
\]
In particular, considering an element $\omega\in \widetilde{\Omega}(\M_{0})$  as a coset on the right-hand side, it is specified by the restriction of any of its representatives to $\beta^{-1}\left( \J^{-1}(0)^{\top} \right)$.
\end{prop}
\begin{proof}
We have the natural surjection
\[
{\Omega}_{\J}(\BLS(M)) \overset{\widetilde{\iota}_{\top}^{*}}{\longrightarrow} \Omega_{\mathrm{bas}\, {S^1}}\left( \beta^{-1}\left( \J^{-1}(0)^{\top} \right) \right) \overset{\left( \beta_{\top}^{-1} \right)^{*}}{\longrightarrow} \Omega_{\mathrm{bas}\, {S^1}}\left( \J^{-1}(0)^{\top} \right) \overset{\left(\pi_{\top}^{*}\right)^{-1}}{\longrightarrow} \widetilde{\Omega}(\M_{0}),
\]
whose kernel is precisely $I_{\J}(\BLS(M))$. Here $\Omega_{\mathrm{bas}\, {S^1}}\left( \J^{-1}(0)^{\top} \right)$ and $\Omega_{\mathrm{bas}\, {S^1}}( \beta^{-1}( \J^{-1}(0)^{\top}))$ are the complexes of basic differential forms on $\J^{-1}(0)^{\top}$ and $\beta^{-1}( \J^{-1}(0)^{\top})$, respectively. 
\end{proof}
Next, observe that there is a natural inclusion
\begin{align*}
\Omega(\M_{0})\hookrightarrow \widetilde{\Omega}(\M_{0}), \qquad \omega &\mapsto \omega.
\end{align*}
Indeed, let $\omega \in \Omega(\M_{0})$ and $\eta \in \Omega(M)$ be such that $\pi_{\top}^{*}\omega=\iota_{\top}^{*}\eta$. Pulling back this equation to $\BLS(M)$ gives  $\beta_{\top}^{*}\pi_{\top}^{*}\omega=\beta_{\top}^{*}\iota_{\top}^{*}\eta=\widetilde{\iota}_{\top}^{*}\beta^{*}\eta$. Hence $\omega$ is a resolution form with smooth extension $\beta^{*}\eta$. This leads to the following 
\begin{definition}
Denote by $C(\M_{0})$ the cokernel of the natural inclusion $\Omega(\M_{0}) \hookrightarrow \widetilde{\Omega}(\M_{0})$, that is, 
\[
C(\M_{0}):=\frac{\widetilde{\Omega}(\M_{0})}{\Omega(\M_{0})}.
\]
\end{definition}
The three differential complexes $\Omega(\M_{0}),\widetilde{\Omega}(\M_{0})$ and $C(\M_{0})$ are naturally related by the short exact sequence of complexes
\begin{center}
\begin{tikzcd}
0 \arrow{r} &\Omega(\M_{0}) \arrow{r} &\widetilde{\Omega}(\M_{0}) \arrow{r} &C(\M_{0}) \arrow{r} &0.
\end{tikzcd}
\end{center}
This induces a long exact sequence in cohomology
\begin{equation}\label{eq:longexact}
\begin{tikzcd}
0 \arrow{r} &H^{0}(\Omega(\M_{0}),d) \arrow{r} &H^{0}(\widetilde{\Omega}(\M_{0}),d) \arrow{r} &H^{0}(C(\M_{0}),d) \arrow[in=175,out=-5,swap]{dll}{\delta} \\
&H^{1}(\Omega(\M_{0}),d) \arrow{r} &H^{1}(\widetilde{\Omega}(\M_{0}),d) \arrow{r} &H^{1}(C(\M_{0}),d) \arrow[in=175,out=-5,swap]{dll}{\delta} \\
&H^{2}(\Omega(\M_{0}),d) \arrow{r} &H^{2}(\widetilde{\Omega}(\M_{0}),d) \arrow{r} &H^{2}(C(\M_{0}),d) \arrow{r} &\ldots.
\end{tikzcd}
\end{equation}
We want to understand this sequence by interpreting each occuring cohomology geometrically. 
 By Sjamaar's theorem \cite[Theorem 5.5]{sjamaar05} one has $H^{*}(\Omega(\M_{0})) \simeq H^{*}(\M_{0})$, which interprets one third of the occurring spaces. Moreover, one has the following
{\begin{prop}\label{resolutionforms}
\begin{enumerate}
\item For $\omega \in {\Omega}_{\J}(\BLS(M))$ the restriction $\omega|_{\widetilde{\J^{-1}(0)}}$ is horizontal.
\item For $\eta \in I_{\J}(\BLS(M))$ the restriction $\eta|_{\widetilde{\J^{-1}(0)}}$ is zero.
\item There is an isomorphism of cochain complexes
\[
\widetilde{\Omega}(\M_{0}) \simeq \Omega(\widetilde{\M}_{0}).
\]
\item If $F \in \F$ is such that  $F\subset \J^{-1}(0)$, then there is a well-defined surjective restriction map $\widetilde{\Omega}(\M_{0})\rightarrow \Omega((E_F\cap\widetilde{\J^{-1}(0)})/{S^1})=\Omega(\widetilde{F})$. 
\end{enumerate}
\end{prop}
\begin{proof} The claims \textit{(1)} and \textit{(2)} are immediate consequences of the density of $\beta^{-1}(\J^{-1}(0)^{\top})$ in $\beta^{-1}(\J^{-1}(0))$. To prove \textit{(3)}, let $\omega$ be a resolution form on $\M_{0}$. Then there is $\eta \in \Omega(\BLS(M))$ such that $\beta_{\top}^{*}\pi_{\top}^{*}\omega=(\iota'_{\top})^{*}\eta$. By restricting this form $\eta$ to $\widetilde{\J^{-1}(0)}$ we obtain an ${S^1}$-basic form which does not depend on the extension $\eta$ of $\omega$ since $\beta^{-1}(\J^{-1}(0)^{\top})$ is dense in $\widetilde{\J^{-1}(0)}$ and $\beta_{\top}$ is a diffeomorphism. Thus we have a natural map
\[
\widetilde{\Omega}(\M_{0}) \rightarrow \Omega(\widetilde{\M}_{0}).
\]
On the other hand, $\widetilde{\J^{-1}(0)}$ is a closed {${S^1}$-invariant} submanifold of $\BLS(M)$ and therefore every {${S^1}$-invariant} differential form on $\widetilde{\J^{-1}(0)}$, and in particular every ${S^1}$-basic form, admits an {${S^1}$-invariant} extension to $\BLS(M)$ and therefore {(since $\widetilde{\J^{-1}(0)}$ contains $\beta^{-1}(\J^{-1}(0)^{\top})$)}  gives us a resolution form on $\M_{0}$ and a map 
\[
\Omega(\widetilde{\M}_{0}) \rightarrow \widetilde{\Omega}(\M_{0}).
\]
These maps are inverse to each other, commute with the differentials, and \textit{(3)} is proved.

We obtain \emph{(4)} by composing the isomorphism from \emph{(3)} with the restriction map $\Omega(\widetilde{\M}_{0})\to \Omega(\widetilde{F})$, noting that Proposition  \ref{prop:desingzero} implies that  $\widetilde{F}$ is a closed submanifold of $\widetilde{\J^{-1}(0)}$. The restriction map is surjective by the general fact that any differential form on a closed submanifold of a smooth manifold can be extended to the ambient manifold. 
\end{proof}}

Thus, the natural map 
\[
H^{*}(\M_{0};\mathbb{R}) \cong H^{*}(\Omega(\M_{0}),d) \longrightarrow H^{*}(\widetilde{\Omega}(\M_{0}),d) \cong H^{*}(\widetilde{\M_{0}})
\]
is in fact induced by the pullback of forms on $\M_{0}$ to forms on $\widetilde{\M_{0}}$ and we have the following

\begin{cor}\label{cor:rauisch}
The complex of resolution forms can equivalently be defined as 
\begin{align*}
\widetilde{\Omega}(\M_{0})&=\left\{ \omega_{0} \in \Omega(\M_{0}^{\top}) \mid \exists\, \widetilde{\rho} \in \Omega_{\mathrm{bas}\,{S^1}}(\widetilde{\J^{-1}(0)}) \colon (\pi_{\top}')^{*}\omega_{0}=(\widetilde{\iota}_{\top})^{*}\widetilde{\rho}  \right\}\\
&=\left\{ \omega_{0} \in \Omega(\M_{0}^{\top}) \mid \exists\, \rho' \in \Omega(\widetilde{\M}_{0}) \colon (\beta_{0}^{\top})^{*}\omega_{0}=(\iota_{0}^{\top})^{*}\rho'  \right\},
\end{align*}
where $\widetilde{\iota}_{\top}$ denotes the inclusion $\beta^{-1}(\J^{-1}(0)^{\top})\rightarrow \widetilde{\J^{-1}(0)}$ and $\iota_{0}^{\top}$ is the inclusion $\widetilde{\M}_{0}^{\top} \rightarrow \widetilde{\M}_{0}$ of
\bqn
\widetilde{\M}_{0}^{\top}:={\widetilde{\pi}\left(\beta^{-1}(\J^{-1}(0)^{\top})\right)=\beta_{0}^{-1}\left(\J^{-1}(0)^{\top}\right)}
\eqn
and $\beta_{0}^{\top} \colon \widetilde{\M}_{0}^{\top} \rightarrow \M_{0}^{\top}$ is the restriction of $\beta_{0}$. 
\end{cor} 
 
So it remains to understand $H^{*}\left( C(\M_{0})\right)$, the cohomology of the cokernel. We start by looking at the restriction maps
\[
 \Omega(\M_{0}) \rightarrow \Omega(F) \quad \text{and} \quad  \Omega(\widetilde{\M}_{0}) \rightarrow \Omega(\widetilde{F}),
\]
which exist by \cite[Lemma 3.3]{sjamaar05} and Prop.\ \ref{resolutionforms}, respectively, and the collection $\mathcal{F}_{0}:=\{ F \subset \J^{-1}(0) \cap M^{S^{1}}\}$ consisting all components of the fixed point set contained in the zero level. We obtain the maps

\begin{align*}
r \colon \Omega(\M_{0}) &\longrightarrow \bigoplus\limits_{F \in \mathcal{F}_{0}}\Omega(F), \qquad
\omega \longmapsto (\omega|_{F})_{F \in \mathcal{F}_{0}}
\end{align*} 
and
\begin{align*}
  \widetilde{r} \colon \Omega(\widetilde{\M}_{0}) &\longrightarrow \bigoplus\limits_{F \in \mathcal{F}_{0}}\Omega(\widetilde{F}), \qquad 
 \omega \longmapsto (\omega|_{\widetilde{F}})_{F \in \mathcal{F}_{0}}.
\end{align*} 
Now consider the following diagram
\begin{center}
\begin{tikzcd}
0 \arrow{r} &\ker(r) \arrow{r} \arrow{d}{(\beta_{0})^{*}_{r}} &\Omega\left(\M_{0}\right) \arrow{d}{(\beta_{0})^{*}} \arrow{r}{{r}} &\bigoplus\limits_{{F \in \mathcal{F}_{0}}}\Omega(F) \arrow{d}{\bigoplus\limits_{{F \in \mathcal{F}_{0}}}(\beta_{0}^{F})^{*}} \arrow{r} &0\\
0 \arrow{r} &\ker(\widetilde{r}) \arrow{r}  &\Omega(\widetilde{\M}_{0}) \arrow{r}{{\widetilde{r}}} &\bigoplus\limits_{{F \in \mathcal{F}_{0}}}\Omega(\widetilde{F}) \arrow{r} &0,
\end{tikzcd}
\end{center}
where $(\beta_{0})^{*}_{r} \colon \ker(r) \rightarrow \ker(\widetilde{r})$ is the restriction of $(\beta_{0})^{*}$ to $\ker(r)$. This diagram is of major interest to us because
\[
C(\M_{0})\simeq\coker((\beta_{0})^{*})
\]
as a consequence of Proposition \ref{resolutionforms}.
By the Snake lemma, there is an exact sequence of the form
\begin{multline*}
0 \rightarrow \ker((\beta_{0})^{*}_{r}) \rightarrow \ker((\beta_{0})^{*}) \rightarrow \ker(\bigoplus\limits_{{F \in \mathcal{F}_{0}}}(\beta_{0}^{F})^{*}) \rightarrow \coker((\beta_{0})^{*}_{r}) \rightarrow \coker((\beta_{0})^{*}) \\\rightarrow \coker(\bigoplus\limits_{{F \in \mathcal{F}_{0}}}(\beta_{0}^{F})^{*}) \rightarrow 0.
\end{multline*}
But $\ker((\beta_{0}^{F})^{*})=0$ because $\beta_{0}^{F}$ is a {surjective} submersion. Thus, we have
\begin{equation} \label{cokernels}
0 \longrightarrow \coker((\beta_{0})^{*}_{r}) \longrightarrow \coker((\beta_{0})^{*}) \longrightarrow \coker(\bigoplus\limits_{{F \in \mathcal{F}_{0}}}(\beta_{0}^{F})^{*}) \longrightarrow 0.
\end{equation}
{We now claim that 
\bq
H^{*}(\coker((\beta_{0})^{*}_{r})) = 0,\label{eq:claim2412525}
\eq
which by the long exact cohomology sequence associated with the short exact sequence \eqref{cokernels} immediately implies that there is an isomorphism
\bq
H^{*}(C(\M_{0})) = H^{*}(\coker((\beta_{0})^{*})) \simeq H^{*}(\coker(\bigoplus\limits_{{F \in \mathcal{F}_{0}}}(\beta_{0}^{F})^{*}))=\bigoplus\limits_{{F \in \mathcal{F}_{0}}} H^{*}(\coker((\beta_{0}^{F})^{*})).\label{eq:usefuliso}
\eq}
{To prove our claim, we first observe that since $\ker((\beta_{0})^{*}_{r})=0$, we have the short exact sequence}
\bq
0 \longrightarrow \ker(r) \stackrel{(\beta_{0})^{*}_{r}}{\longrightarrow} \ker(\widetilde{r}) \longrightarrow \coker((\beta_{0})^{*}_{r}) \longrightarrow 0.\label{eq:ses}
\eq
By 
\cite[Theorem 11.1.1]{guillemin-sternberg99} we have an isomorphism
\[
H^{*}(\ker(\widetilde{r}))\simeq H^{*}_{c}\left(\widetilde{\M}_{0} \setminus \left(\bigcup\limits_{F \in \mathcal{F}_{0}}\widetilde{F}\right)\right),
\]
where the right-hand side denotes the cohomology of the complex $\Omega_{c}(\widetilde{\M}_{0} \setminus (\cup_{F \in \mathcal{F}_{0}}\widetilde{F}))$ formed by the differential forms on $\widetilde{\M}_{0} \setminus (\cup_{F \in \mathcal{F}_{0}}\widetilde{F}))$  with compact support, and the map $\Omega_{c}(\widetilde{\M}_{0} \setminus (\cup_{F \in \mathcal{F}_{0}}\widetilde{F}))\to \Omega(\widetilde{\M}_{0})$ inducing the above isomorphism is extension by zero, compare \cite[Section 11.1]{guillemin-sternberg99}.  {This implies, by the long exact sequence associated with \eqref{eq:ses},  that our claim  \eqref{eq:claim2412525} follows from the following Proposition, taking into account that $\beta_{0}$ is a diffeomorphism away from the exceptional loci.}
\begin{prop}\label{prop:ext}
Extension by zero induces an isomorphism 
\[
\mathrm{ext} \colon H^{*}_{c}\left(\M_{0} \setminus \left(\bigcup\limits_{F \in \mathcal{F}_{0}}F\right)\right) \rightarrow H^{*}(\ker(r)).
\]
\end{prop}
{For the proof we need the following basic}
 \begin{lem}\label{lem:8592306}
 Let $F \subset M$ be a component of the fixed point set, $\pi \colon U \rightarrow F$ be an invariant tubular neighborhood and $i \colon F \rightarrow U$ be the inclusion. If $\omega \in \Omega^{k}_{\mathrm{bas}\,{S^1}}(U)$ is closed and $i^{*}\omega=0$, there is $\nu \in \Omega_{\mathrm{bas}\,{S^1}}^{k-1}(U)$ with $i^{*}\nu=0$ and $d\nu=\omega$.
 \end{lem}
\begin{proof}[Proof of Lemma \ref{lem:8592306}]
Since $U$ is an equivariant deformation retraction, we have isomorphisms
\[
\pi^{*} \colon H^{*}_{\mathrm{bas}\,{S^1}}(F) \rightarrow H^{*}_{\mathrm{bas}\,{S^1}}(U) \qquad i^{*} \colon H^{*}_{\mathrm{bas}\,{S^1}}(U) \rightarrow H^{*}_{\mathrm{bas}\,{S^1}}(F).
\]
For $i^{*}\omega=0$ it follows, that $[i^{*}\omega]=0 \in H^{*}_{\mathrm{bas}\,{S^1}}(F)$. Since $i^{*}$ is an isomorphism, we find that $[\omega]=0 \in H^{*}_{\mathrm{bas}\,{S^1}}(U)$ and there is $\nu' \in \Omega^{k-1}_{\mathrm{bas}\,{S^1}}(U)$ with $d\nu'=\omega$. Moreover, $0=i^{*}\omega =i^{*}d\nu'=di^{*}\nu'$ and $i^{*}\nu'$ is closed. Then 
\[
\nu:=\nu'-\pi^{*}i^{*}\nu'
\]
is such that
\begin{itemize}
\item $d\nu =d(\nu'-\pi^{*}i^{*}\nu')=d\nu'-\pi^{*}(di^{*}\nu')=d\nu'=\omega$,
\item $i^{*}\nu=i^{*}(\nu'-\pi^{*}i^{*}\nu')=i^{*}\nu'-i^{*}\pi^{*}i^{*}\nu'=i^{*}\nu'-(\underbrace{\pi\circ i}_{=id})^{*}i^{*}\nu'=0$.
\end{itemize}
\end{proof}
\begin{proof}[Proof of Proposition \ref{prop:ext}]
\begin{enumerate}
\item \textbf{Well-definedness:} At first, we have to make sure that {the extension map is} well-defined {with the specified target. To this end, we} use the following identification from Section \ref{sec:2.1}: 
\[
\Omega(\M_{0}) \cong \frac{\Omega_{\J}(M)}{I_{\J}(M)},
\]
where
\[
\Omega_{\J}(M):=\left\{ \omega \in \Omega(M)^{{{S^1}}} \mid \omega|_{\J^{-1}(0)^{\top}} \text{ basic} \right\}
\]
and
\[
I_{\J}(M):=\left\{ \omega \in \Omega(M)^{{{S^1}}} \mid \omega|_{\J^{-1}(0)^{\top}} =0 \right\}.
\]
Then 
\[
\ker(r) \cong \frac{\left\{ \omega \in \Omega(M)^{{{S^1}}} \mid \omega|_{\J^{-1}(0)^{\top}} \text{ basic}, \, \omega|_{F}=0 \, \text{for all} \, F \in \mathcal{F}_{0} \right\}}{\left\{ \omega \in \Omega(M)^{{{S^1}}} \mid \omega|_{\J^{-1}(0)^{\top}} =0, \, \omega|_{F}=0 \, \text{for all} \, F \in \mathcal{F}_{0}   \right\}}.
\]
Thus, extension by zero induces a well-defined map 
\[
\mathrm{ext} \colon \Omega_{\textrm{bas }{S^1}}(\J^{-1}(0)^{\top})_{c} \longrightarrow \ker(r),
\]
{and since this map commutes with the differentials it descends to cohomology.}

\item \textbf{Surjectivity:}
Let $\omega \in \Omega^{k}(M)^{{S^1}}$ be such that 
\begin{itemize}[leftmargin=6pc]
\item $\omega|_{\J^{-1}(0)^{\top}}$ is basic,
\item $\omega|_{F}=0$ for all $F \in \mathcal{F}_{0}$,
\item $d \omega =0$.
\end{itemize}
{For $F\in \F_0$, let} $\pi \colon U \rightarrow F$ be an invariant tubular neighborhood and $i \colon F \rightarrow U$ {the inclusion}. By the above lemma, we find $\nu_{F} \in \Omega_{\mathrm{bas}\,{S^1}}^{k-1}(U)$ such that 
\[
\omega = d\nu_{F} \text{ on $U$} \qquad \nu|_{F}=0.
\]
Now, let $\varrho_{F} \in C_{c}^{\infty}(U)^{{S^1}}$ be a ${S^1}$-invariant compactly supported function, which is equal to $1$ on a neighborhood of $F$.
Then 
\[
\widetilde\omega:=\left(\omega-\sum\limits_{F \in \mathcal{F}_{0}}d(\varrho_{F} \cdot \nu_{F})\right)\Bigg|_{\J^{-1}(0)^{\top}} \in \Omega^{k}_{\mathrm{bas}\, {S^1}}\left(\J^{-1}(0)^{\top}\right)_{c}
\]
{is well-defined and} such that
\begin{itemize}
\item $d\widetilde\omega=0$,
\item $\widetilde\omega$ is basic,

\end{itemize}
{which shows that $\mathrm{ext}(\widetilde\omega)\in \ker(r)$, and by construction we have $[\mathrm{ext}(\widetilde\omega)]=[\omega]\in H^{*}(\ker(r))$.}

\item \textbf{Injectivity:}
Let $\omega \in \Omega^{k}_{\mathrm{bas}\,{S^1}}(\J^{-1}(0)^{\top})_{c}$ and $\nu' \in \Omega^{k-1}(M)^{{S^1}}$ be such that
\[
{\mathrm{ext}(\omega)}=d{\nu'}, \qquad \nu'|_{\J^{-1}(0)^{\top}} \, \text{{is }basic}, \qquad \nu'|_{F}=0   \text{ for all} \, F \in \mathcal{F}_{0}.
\]
Then{, for each $F\in \F_0$,} there is an invariant tubular neighborhood $U$ of $F$ on which $\mathrm{ext}(\omega)|_{U}=0$, {which implies that $\nu'|_{U}$ is closed}. By the above lemma there exists $\alpha_{F} \in \Omega_{\mathrm{bas}\,{S^1}}^{k-2}(U)$ such that
\[
\nu'|_{U} =d\alpha_{F} \qquad \alpha|_{F}=0.
\]
Now
\[
\nu:=\left(\nu'-\sum\limits_{F \in \mathcal{F}_{0}}d(\varrho_{F} \cdot \alpha_{F})\right)\Bigg|_{\J^{-1}(0)^{\top}}
\]
is an element of $\Omega_{\mathrm{bas}\,{S^1}}(\J^{-1}(0)^{\top})_{c}$ and ${\mathrm{ext}(\omega)}=d\nu.$
%
\end{enumerate}
\end{proof}
{Having proved \eqref{eq:claim2412525},} we deduce by applying Proposition \ref{exceptionalcohomology}  that 
 \begin{align*}
 H^{*}(C(\M_{0})) &\simeq \bigoplus\limits_{{F \in \mathcal{F}_{0}}} H^{*}(\coker((\beta_{0}^{F})^{*}))=\bigoplus\limits_{{F \in \mathcal{F}_{0}}} \coker\big(H^{*}(F) \stackrel{{(\beta_{0}^{F})^{*}}}{\longrightarrow} H^{*}(\widetilde{F}) \big)\\
 &=\bigoplus\limits_{{F \in \mathcal{F}_{0}}} \coker\left(H^{*}(F) \rightarrow H^{*}(F) \otimes {\frac{\mathbb{R}[\widetilde{\sigma}_{0}|_{\widetilde F},\widetilde{\Xi}|_{\widetilde F}]}{I_F}} \right)\\
 &\simeq \bigoplus\limits_{{F \in \mathcal{F}_{0}}} H^{*}(F) \otimes {\frac{\mathbb{R}[\widetilde{\sigma}_{0}|_{\widetilde F},\widetilde{\Xi}|_{\widetilde F}]_{\geq 1}}{I_F}},
\end{align*}
{where $\mathbb{R}[\widetilde{\sigma}_{0}|_{\widetilde F},\widetilde{\Xi}|_{\widetilde F}]_{\geq 1}$ denotes the ideal in the polynomial ring $\mathbb{R}[\widetilde{\sigma}_{0}|_{\widetilde F},\widetilde{\Xi}|_{\widetilde F}]$ given by all polynomials of degree $\geq 1$, we used in the first line that $\ker((\beta_{0}^{F})^{*})=0$ on forms since $\beta_{0}^{F}$ is a surjective submersion,} and in the last line that {$(\beta_{0}^{F})^{*}$ embeds  $H^{*}(F)$ into} $H^{*}(F) \otimes {\frac{\mathbb{R}[\widetilde{\sigma}_{0}|_{\widetilde F},\widetilde{\Xi}|_{\widetilde F}]}{I_F}}$ {as the degree zero subspace}.  
This yields a full understanding of the long exact sequence {\eqref{eq:longexact}:}
\begin{thm}\label{longexactsequence}
There is a long exact sequence of the form
\[
\ldots \longrightarrow H^{k}\left(\M_{0}\right) \longrightarrow H^{k}(\widetilde{\M}_{0}) \longrightarrow \left(\bigoplus\limits_{{F \in \mathcal{F}_{0}}} {H^{*}(F) \otimes\frac{\mathbb{R}[\widetilde{\sigma}_{0}|_{\widetilde F},\widetilde{\Xi}|_{\widetilde F}]_{\geq 1}}{I_F}}\right)_{k} \longrightarrow H^{k+1}(\widetilde{\M}_{0}) \longrightarrow \ldots.
\]{\qed}
\end{thm}
As a special case we get
\begin{cor}\label{splitting}
If the fixed point set {$\J^{-1}(0) \cap M^{S^{1}}$} consists only of isolated fixed points and {one has}  $H^{2k+1}(\M_{0})=0$ for all $k$, {then the natural map $
H^{\ast}(\M_{0})\to H^{\ast}(\widetilde{\M}_{0})
$ is injective and there is}  a (non-canonical) {isomorphism}
\[
H^{*}(\widetilde{\M}_{0})\cong H^{*}\left(\M_{0}\right) \oplus  \bigoplus\limits_{{F \in \mathcal{F}_{0}}} {\frac{\mathbb{R}[\widetilde{\sigma}_{0}|_{\widetilde F},\widetilde{\Xi}|_{\widetilde F}]_{\geq 1}}{I_F}}.
\]
\end{cor}

\begin{rem}\label{singularinresolution}
This corollary gives us a way to identify those classes in $H^{*}(\widetilde{\M_{0}})$ which come from classes in $H^{*}(\M_{0})$. In fact, consider the natural map 
\[
R \colon H^{*}(\widetilde{\M_{0}}) \longrightarrow \bigoplus\limits_{F \in \mathcal{F}_{0}}H^{*}(\widetilde{F}) \longrightarrow \bigoplus\limits_{F \in \mathcal{F}_{0}}\coker(H^{*}(F) \rightarrow H^{*}(\widetilde{F})),
\]
induced by restriction to the exceptional components and projection to the according cokernels.
Then, we have 
\[
H^{*}(\M_{0}) \cong \left\{[\omega] \in H^{*}(\widetilde{\M_{0}}) \mid R([\omega])=0 \right\}.
\]
\end{rem}

\begin{proof}
Since {$H^{*}(F)=\R$ when $F$ is a point and}  the generators {$\widetilde{\sigma}_{0}|_{\widetilde F},\widetilde{\Xi}|_{\widetilde F}$} are of degree $2$, and in particular even, the {odd degree components of $\bigoplus\limits_{{F \in \mathcal{F}_{0}}} H^{*}(F) \otimes\frac{\mathbb{R}[\widetilde{\sigma}_{0}|_{\widetilde F},\widetilde{\Xi}|_{\widetilde F}]_{\geq 1}}{I_F}=\bigoplus\limits_{{F \in \mathcal{F}_{0}}} \frac{\mathbb{R}[\widetilde{\sigma}_{0}|_{\widetilde F},\widetilde{\Xi}|_{\widetilde F}]_{\geq 1}}{I_F}$} 
 vanish and the long exact sequence from Theorem \ref{longexactsequence} breaks down into sequences
\[
0 \rightarrow H^{2k}\left(\M_{0}\right) \rightarrow H^{2k}(\widetilde{\M}_{0}) \rightarrow {\left(\bigoplus\limits_{{F \in \mathcal{F}_{0}}} \frac{\mathbb{R}[\widetilde{\sigma}_{0}|_{\widetilde F},\widetilde{\Xi}|_{\widetilde F}]_{\geq 1}}{I_F}\right)_{2k}}\rightarrow H^{2k+1}\left(\M_{0}\right) \rightarrow H^{2k+1}(\widetilde{\M}_{0}) \rightarrow 0,
\]
{where we denote by the subindex $2k$ the subspace of degree $2k$ elements}. 
Now the vanishing of the odd cohomology of $\M_{0}$ implies that also the odd cohomology of $\widetilde{\M}_{0}$ vanishes and the above sequence further reduces to 
\[
0 \longrightarrow H^{k}\left(\M_{0}\right) \longrightarrow H^{k}(\widetilde{\M}_{0}) \longrightarrow {\left(\bigoplus\limits_{{F \in \mathcal{F}_{0}}} \frac{\mathbb{R}[\widetilde{\sigma}_{0}|_{\widetilde F},\widetilde{\Xi}|_{\widetilde F}]_{\geq 1}}{I_F}\right)_k} \longrightarrow 0,
\]
where in case of odd $k$ all the vector spaces in the sequence are $0$.
But every short exact sequence of vector spaces is split by \cite[Proposition 4.3]{maclane95} and the corollary follows.
\end{proof}

More generally we obtain   
\begin{cor}\label{generalF}
Suppose that for some $k\in \N$ and each {connected} component $F \subset{\J^{-1}(0) \cap M^{S^{1}}}$  we have $H^{{2k-1}}(F)=0$. Then the natural map
$
H^{2k}(\M_{0})\to H^{2k}(\widetilde{\M}_{0})
$ 
is injective.
\end{cor}
\begin{proof}
Again, since $H^{{2k-1}}(F)=0$ our long exact sequence from Theorem \ref{longexactsequence} {looks like}
\[
 {\cdots\rightarrow \left(\bigoplus\limits_{{F \in \mathcal{F}_{0}}}H^{*}(F) \otimes \frac{\mathbb{R}[\widetilde{\sigma}_{0}|_{\widetilde F},\widetilde{\Xi}|_{\widetilde F}]_{\geq 1}}{I_F}\right)_{2k-1}=0 \rightarrow H^{2k}\left(\M_{0}\right) \rightarrow H^{2k}(\widetilde{\M}_{0}) \rightarrow \cdots .}
\]
{Thanks to the exactness, the map $H^{2k}\left(\M_{0}\right) \rightarrow H^{2k}(\widetilde{\M}_{0})$ is injective.} 
\end{proof}
\begin{rem} \label{rem:notinj}  {Note that for finite-dimensional vector spaces $A$ and $B$, the statement that there is  a finite-dimensional vector space $V$ such that $B=A\oplus V$ is equivalent to the statement that $A$ injects into $B$.} In the situation of Corollary \ref{generalF}, the injectivity need not hold in odd degrees, even if $H^k(F)=0$ for all $k\in \N$ and $F\in \F_0$, as will become apparent later in Example \ref{4spheres}. Note that also from the general theory of resolution of singularities we cannot expect that {the natural map $H^{*}(\M_{0}) \longrightarrow H^{*}(\widetilde{\M}_{0})$} is always injective, since resolutions do not always induce injective maps in cohomology as the example of the nodal curve shows, see \cite[Example I.4.9.1]{hartshorne77}. In fact, the nodal curve has non trivial cohomology in positive degrees while the resolution is contractible. 
\end{rem}

\begin{rem}\label{canonicity}
 The isomorphism {$H^{*}(\widetilde{\M}_{0})\cong H^{*}\left(\M_{0}\right) \oplus  \bigoplus_{F \in \mathcal{F}_{0}} \frac{\mathbb{R}[\widetilde{\sigma}_{0}|_{\widetilde F},\widetilde{\Xi}|_{\widetilde F}]_{\geq 1}}{I_F}$ from Corollary \ref{splitting}} can be made explicit by recalling how short exact sequences of vector spaces split, compare \cite[Propositions 4.2 and 4.3]{maclane95}.
 More precisely, 
the splitting isomorphism from Corollary \ref{splitting} and a projection 
\[
H^{*}(\widetilde{\M}_{0}) \longrightarrow H^{*}(\M_{0})
\]
can be made explicit by fixing a basis of {$\bigoplus_{F \in \mathcal{F}_{0}} \frac{\mathbb{R}[\widetilde{\sigma}_{0}|_{\widetilde F},\widetilde{\Xi}|_{\widetilde F}]_{\geq 1}}{I_F}$} and finding preimages of the basis vectors inside $H^{*}(\widetilde{\M}_{0})$.
\end{rem}

\begin{example}\label{2spheres}
As a particularly easy example consider the product $S^{2} \times S^{2}$ with its product symplectic form where we equipped each sphere with its standard volume form. Then the $S^{1}$ action defined by diagonally rotating around the $z$-axis with speed $1$ is a Hamiltonian action with momentum map
\begin{align*}
\J \colon S^{2} \times S^{2} &\longrightarrow \mathbb{R}, \qquad ((x_{i},y_{i},z_{i}))_{i=1,2} \longmapsto z_{1}+z_{2}.
\end{align*}
Now, zero is not a regular value of $\J$ as the zero level contains the two fixed points $((0,0,1),(0,0,-1))$ and $((0,0,-1),(0,0,1))$. In fact, the zero level set is a suspended $2$-torus and the symplectic quotient $\M_{0}$ is a suspension of $S^{1}$ and in particular homeomorphic to $S^{2}$. From Proposition \ref{bisphere} we see that in this case all the occuring exceptional divisors of the partial desingularization are in fact points, as the occuring spheres $S_{F}^{\pm}$ in the weight spaces are one-dimensional. Thus, for all fixed points $F$ in the zero level set we have $F=\widetilde{F}$ and the long exact sequence from Theorem \ref{longexactsequence} shows
\[
\beta^{*} \colon H^{k}(\M_{0}) \cong H^{k}(\widetilde{\M}_{0}) \qquad \forall k \in \mathbb{N}
\]
and resolution cohomology and ordinary cohomology of $\M_{0}$ are isomorphic. Moreover, using techniques from toric geometry, one shows that the curvature class of the partial desingularization $\widetilde{\M_{0}}$ vanishes in this case.  In fact, $M=S^{2} \times S^{2}$ is a toric symplectic manifold acted upon by the two-torus $T^{2}=S^{1} \times S^{1}$ where each circle rotates one of the spheres around the $z$-axis. A momentum map for this action is given by
\begin{align*}
\J' \colon M &\longrightarrow \mathbb{R}^{2}, \qquad ((x_{i},y_{i},z_{i}))_{i=1,2} \longmapsto (z_{1},z_{2}),
\end{align*} 
where we identified the dual of the Lie algebra of $T^{2}$ with $\mathbb{R}^{2}$. Our original momentum map $\J$ arises from $\J'$ by composing it with the map
\begin{align*}
\mathbb{R}^{2} \longrightarrow \mathbb{R}, \qquad (x,y)  \longmapsto x+y
\end{align*}
as we consider the diagonal $S^{1}$-action. Now, the momentum image of $\J'(M) \subset \mathbb{R}^{2}$ is the convex polytope
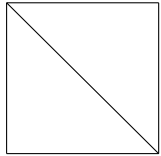
\begin{center}
\begin{tikzpicture}
\draw (0,0) -- (2,0);
\draw (2,0) -- (2,2);
\draw (2,2) -- (0,2);
\draw (0,2) -- (0,0);

\draw (0,2) -- (2,0);
\end{tikzpicture}
\end{center}
where the drawn anti-diagonal depicts $\J'(\J^{-1}(0))$. From the work of Guillemin--Sternberg, see \cite[Theorem 6.1]{guillemin-sternberg89}, we already know, that $\M_{0}$ will be a symplectic manifold, even though zero is not a regular value as all fixed point on $M$ have either index or coindex $2$, on which the complementary $S^{1} \subset T^{2}, z \mapsto (z,1)$ acts in a Hamiltonian way with momentum map $\J''$ whose momentum image we obtain by the projection the anti-diagonal as  

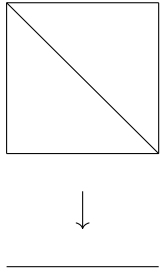
\begin{center}
\begin{tikzpicture}
\draw (0,0) -- (2,0);
\draw (2,0) -- (2,2);
\draw (2,2) -- (0,2);
\draw (0,2) -- (0,0);

\draw (0,2) -- (2,0);

\draw[->] (1,-0.5) -- (1,-1);

\draw (0,-1.5) -- (2,-1.5);
\end{tikzpicture}
\end{center}
Now, $\M_{0}$ is a toric symplectic two-manifold whose momentum polytope associated to the $S^{1}$-action is an interval. Thus, $\M_{0}$ has to be $S^{2}$ by the classification of symplectic toric manifolds due to Delzant, see \cite[Theorem IV.4.20 and Section VII.2]{audin2004}, \cite[Chapter 1]{guillemin94} and \cite[Theorem 28.2]{cannas-da-silva}. A similar reasoning applies to the blow-up $\BLS^{\mathbb{C}}$ whose momentum polytope is

\begin{center}
\begin{tikzpicture}
\draw (0.25,0) -- (1.75,0);
\draw (1.75,0) -- (2,0.25) ;

\draw (2,0.25) -- (2,1.75);
\draw (2,1.75) -- (1.75,2);
\draw (1.75,2) -- (0.25,2);
\draw (0.25,2) -- (0,1.75);
\draw (0,1.75) -- (0,0.25);
\draw (0,0.25) -- (0.25,0);
\end{tikzpicture}
\end{center}
by \cite[Remark 1.5]{lerman95}. Now, we may depict the strict transform as the anti-diagonal in

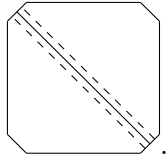
\begin{center}
\begin{tikzpicture}
\draw (0.25,0) -- (1.75,0);
\draw (1.75,0) -- (2,0.25) ;
\draw (2,0.25) -- (2,1.75);
\draw (2,1.75) -- (1.75,2);
\draw (1.75,2) -- (0.25,2);
\draw (0.25,2) -- (0,1.75);
\draw (0,1.75) -- (0,0.25);
\draw (0,0.25) -- (0.25,0);

\draw (1.875,0.125) -- (0.125,1.875);
\draw[dashed] (1.9375,0.1875) -- (0.1875,1.9375);
\draw[dashed] (1.8125,0.0625) -- (0.0625,1.8125);
\end{tikzpicture}.
\end{center}
The dashed part represents a neighborhood of $\widetilde{\J}^{-1}(0)$ in $\BLS^{\mathbb{C}}(M)$ and the image tells us that it is $T^{2}$-equivariantly symplectomorphic to a neighborhood of $S^{2} \times S^{1}$ in $S^{2} \times S^{2}$ as depicted in 
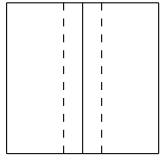
\begin{center}
\begin{tikzpicture}
\draw (0,0) -- (2,0);
\draw (2,0) -- (2,2);
\draw (2,2) -- (0,2);
\draw (0,2) -- (0,0);

\draw (1,0) -- (1,2);
\draw[dashed] (0.75,0) -- (0.75,2);
\draw[dashed] (1.25,0) -- (1.25,2);
\end{tikzpicture}
\end{center}
Furthermore,  the bundle $\widetilde{\J}^{-1}(0) \rightarrow \widetilde{\M_{0}}$ is trivial.
\end{example}

\begin{example}\label{3sphere}
Moving on from the previous example to the case of the standard diagonal $S^{1}$-action on $M:=S^{2} \times S^{2} \times S^{2}$ with momentum map 
\begin{align*}
\J \colon S^{2} \times S^{2} \times S^{2} &\longrightarrow \mathbb{R}, \qquad ((x_{i},y_{i},z_{i}))_{i=1,2,3} \longmapsto z_{1}+z_{2}+z_{3}+1
\end{align*}
we see that the critical values are $-2,0,2$ and $4$. There are $3$ isolated fixed points in the $0$-level set and the isotropy representation at each of those has one positive and two negative weights. Again, using toric geometry and the fact that $M$ is indeed a toric symplectic manifold, one proves that $\M_{0}$ is homeomorphic to $\CP^{2}$ and $\widetilde{\M_{0}}$ is homeomorphic to the blow-up of $\M_{0}$ in three points. Moreover, the excetional bundles $\widetilde{F}$ are equal to $\CP^{1}$s in this case and the curvature of $\M_{0}$ does not vanish. Thus, in this examples all terms in the long exact sequence from Theorem \ref{longexactsequence} are excplicitely known. In fact, we consider $M$ as a toric Hamiltonian $T^{3}$-mainfold, whose momentum polytope is a cube. By arguments similar to those in the case of two spheres, one obtains that the momentum image of a complementary $T^{2}$-action on the smooth quotient $\M_{0}$ is 
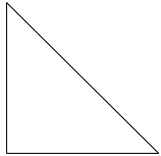
\begin{center}
\begin{tikzpicture}
\draw (0,0) -- (2,0);
\draw (2,0) -- (0,2);
\draw (0,2) -- (0,0);
\end{tikzpicture}
\end{center}
and $\M_{0}$ is symplectomorphic to $\CP^{2}$ while the partial desingularization $\widetilde{\M_{0}}$ has momentum image
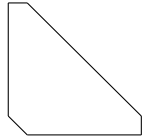
\begin{center}
\begin{tikzpicture}
\draw (0.25,0) -- (1.75,0);
\draw (1.75,0) -- (1.75,0.25);
\draw (1.75,0.25) -- (0.25,1.75);
\draw (0.25,1.75)-- (0,1.75) ;
\draw (0,1.75) -- (0,0.25);
\draw (0,0.25) -- (0.25,0);
\end{tikzpicture}
\end{center}
and is therefore symplectomorphic to $\CP^{2}$ blown-up in the three corner points. Now the bundle $\widetilde{\J}^{-1}(0) \rightarrow \widetilde{\M_{0}}$ cannot be trivial since it is equal to the Hopf bundle $S^{3} \rightarrow \CP^{1}$, when restricted to the exceptional loci of the blow-up
\[
\mathrm{Bl}^{\mathbb{C}}_{3}(\CP^{2}) \longrightarrow \CP^{2}
\]
and the curvature class of $\widetilde{\M_{0}}$ does not vanish.
\end{example}

\section{Singular Kirwan surjectivity}\label{sec:singkirwansurj}

Let $(M,\sigma)$ be a connected symplectic manifold carrying a Hamiltonian group action of a  compact Lie group $G$ with equivariant momentum map $\J \colon M \rightarrow \mathfrak{g}^{*}$ and consider the associated {symplectic quotient} $\M_0:=\J^{-1}(0)/G$. One of the main tools in the study of its cohomology is the \emph{Kirwan map}, which in the case when $0$ is a regular value of the momentum map is defined as the composition 
\begin{center}
\begin{tikzcd}
\kappa: \, H^{*}_{G}(M) \arrow{r}{\iota^{*}} &H^{*}_{G}(\J^{-1}(0)) \arrow{r}{\Car} &H^{*}_{\mathrm{bas} \, G}(\J^{-1}(0)) \arrow{r}{(\pi^{*})^{-1}} &H^{*}(\M_{0}),
\end{tikzcd}
\end{center}
where $\iota: \J^{-1}(0) \rightarrow M$ and $\pi:\J^{-1}(0) \rightarrow \M_0$ are the natural injection and projection, respectively, 
and $\Car \colon H^{*}_{G}(\J^{-1}(0)) \rightarrow H^{*}_{\mathrm{bas} \, G}(\J^{-1}(0))$  is the \emph{Cartan isomorphism} of equivariant and basic cohomology. For $G=S^1$ it is explicitely given by \cite{duistermaat94}
\[
\Car\left(\sum\limits_{I} \omega_{I} \cdot x^{I}\right):= \sum\limits_{I} \omega_{I} \wedge \Xi^{I} - \Theta \wedge \sum\limits_{I} i_{\overline{X}}\omega_{I} \wedge \Xi^{I},
\]
where $\Theta \in \Omega^{1}(M)$ is a connection form, $\Xi$ its curvature and $X$ a generator of $\mathfrak{g}$ with $\Theta(\overline{X})=1$. As a major application of our theory of resolution differential forms developed in the previous section,  in this section we shall extend the above definition of the Kirwan map to the case where $0$ is not necessarily a regular value of $\J$ and study its surjectivity.

 \subsection{The resolution Kirwan map}  To begin, we look at our complex of resolution differential forms in the light of $\mathfrak{g}$-differential graded algebras, see \cite[Section 4]{goertsches-nozawa-toeben12} and \cite{guillemin-sternberg99} for a systematic exposition. 
\begin{definition}
Let $\mathfrak{g}$ be a finite-dimensional Lie algebra and $A=\bigoplus A_{k}$ a $\mathbb{Z}$-graded algebra. We call $A$ a \emph{$\mathfrak{g}$-differential graded algebra ($\mathfrak{g}$-dga)} if there is a derivation $d \colon A \rightarrow A$ of degree $1$, together with derivations $i_{X} \colon A \rightarrow A$ of degree $-1$ and $L_{X} \colon A \rightarrow A$ of degree $0$ for all $X \in \mathfrak{g}$, $i_{X}$ and $L_{X}$ depending linearly on $X$, such that
\begin{multicols}{2}
\begin{enumerate}
\item $d^{2}=0$,
\item $i_ {X}^{2}=0$,
\item $[L_{X},L_{Y}]=L_{[X,Y]}$,
\item $[L_{X},i_{Y}]=i_{[X,Y]}$,
\item $[d,L_{X}]=0$,
\item $L_{X}=d i_{X}+i_{X}d$.
\end{enumerate} 
\end{multicols}
\end{definition}
As an example, we can think about $A$ being the complex of differential forms on a $G$-manifold, $d$ the exterior differential, $i_{X}$ contraction with the fundamental vector field of $X$ and $L_{X}$ the Lie derivative in direction of the fundamental vector field. From such a $\mathfrak{g}$-dga $A$ one forms two new complexes in a standard manner, namely the \emph{Cartan complex of $A$}
\[
C_{G}(A):=\left( S(\mathfrak{g^{*}})\otimes A \right)^{\mathfrak{g}}
\]
and the \emph{basic subcomplex of $A$}
\[
A_{\mathrm{bas} \, \mathfrak{g}}:=\{\omega \in A \mid i_{X}\omega=0 \: \text{and} \: L_{X}\omega=0 \: \, \forall X \in \mathfrak{g}\}.
\]
In our setting, we introduce the $\mathfrak{g}$-dgas 
\begin{align}\label{eq:dgas}
\begin{split}
\Omega^{*}(\J^{-1}(0))&:=\left\{ \omega \in \Omega(\J^{-1}(0)^{\top}) \mid \exists \eta \in \Omega(M) \colon \iota_{\top}^{*} \eta = \omega \right\}, \\
\widetilde{\Omega}^{*}(\J^{-1}(0)):&=\left\{ \omega \in \Omega(\J^{-1}(0)^{\top}) \mid \exists \eta \in \Omega(\BLS (M)) \colon (\iota'_{\top})^{*} \eta = \beta_{\top}^{*}\omega \right\}.
\end{split}
\end{align}
As it turns out, the complexes of differential and resolution  forms on $\M_{0}$ are isomorphic to the basic subcomplexes of these $\mathfrak{g}$-dgas, respectively,  via the pullback associated with the quotient map $\pi_{\top} \colon J^{-1}(0)^{\top} \rightarrow \M_{0}^{\top}$ because $S^1$ is connected. The associated 
equivariant cohomology groups are
\bqn
H^{*}_{{S^1}}(\J^{-1}(0)):= H^\ast(C_{S^1}(\Omega^\ast(\J^{-1}(0))),d_{S^1}), \qquad H^\ast(C_{S^1}(\widetilde \Omega^\ast(\J^{-1}(0))),d_{{S^1}}),
\eqn
 respectively, where $d_{S^1}$ is the equivariant differential. The crucial difference between $\Omega^{*}(\J^{-1}(0))$ and $\widetilde{\Omega}^{*}(\J^{-1}(0))$ and the driving force behind our investigations is that $\Omega^{*}(\J^{-1}(0))$ is not invariant under multiplication with a connection form, while $\widetilde{\Omega}^{*}(\J^{-1}(0))$ is. This makes $\widetilde{\Omega}^{*}(\J^{-1}(0))$ a $W^{*}$-module, and even a $\mathfrak{g}$-dga of type (C), in the sense of \cite[Definition 3.4.1]{guillemin-sternberg99} and \cite[Definition 2.3.4]{guillemin-sternberg99}. This is a powerful property, because for such a $W^{*}$-module $A$, the map
\begin{align}
\label{eq:22.04.2022}
A_{\mathrm{bas} \, \mathfrak{g}} &\rightarrow C_{G}(A),\qquad \omega \mapsto 1 \otimes \omega,
\end{align}
 induces an isomorphism in cohomology with homotopy inverse given by the \emph{Cartan map} 
 $$\Car \colon C_{G}(A) \rightarrow A_{\mathrm{bas} \, \mathfrak{g}},$$ 
 see \cite[Sections 4 and 5]{guillemin-sternberg99}. In fact, there it is proved that $\Car$ is a quasi-isomorphism. But clearly, $\Car(1\otimes\omega)=\omega$ for $\omega \in  A_{\mathrm{bas} \, \mathfrak{g}}$, so (\ref{eq:22.04.2022}) induces the inverse to the Cartan map in cohomology. We are now ready to define the resolution Kirwan map
\bq\label{eq:20.04.2022}
\mathcal{K} \colon H^{\ast}_{{S^1}}(M) \rightarrow H^{*}(\widetilde{\Omega}^{*}(\M_{0}),d)
\eq
from the equivariant cohomology $H^{\ast}_{{S^1}}(M)$ to the resolution cohomology $H^{*}(\widetilde{\Omega}^{*}(\M_{0}),d)$.
 
 \begin{definition}\label{def:20.04.2022}  The  \emph{resolution Kirwan map} $\mathcal K$ is defined as the composition of maps  \\
\begin{center}
\begin{tikzcd}
H^{*}_{{S^1}}(M) \arrow{r}{\iota_{\top}^{*}} \arrow{ddrr}{\mathcal{K}} & H^{*}_{{S^1}}\left(\J^{-1}(0)\right) \arrow{r}{\mathrm{inc}} &H^{*}\Big(C_{S^1}\big (\widetilde{\Omega}^{*}\left(\J^{-1}(0)\big )\right),d_{{S^1}}\Big) \arrow{d}{\Car} \\ 
&&  H^{*}\left(\widetilde{\Omega}^{*}\left(\J^{-1}(0)\right)_{\mathrm{bas} \, \mathfrak{g}},d\right) \arrow{d}{\left(\pi_{\top}^{*}\right)^{-1}} \\
&& H^{*}(\widetilde{\Omega}^{*}(\M_{0}),d) 
\end{tikzcd}
\end{center}
\end{definition}
\begin{rem}\label{rem:15.07.2023}
Another way to view the resolution Kirwan map follows from Proposition \ref{resolutionforms}, Corollary \ref{TransformZeroSet} and Lemma \ref{EquivTransform}.  From this perspective, the resolution Kirwan map $\mathcal{K}$ is equal to the composition of $\beta^{*} \colon H^{*}_{{S^1}}(M) \rightarrow H^{*}_{{S^1}}(\BLS(M))$ with the regular Kirwan map of the blow-up $\kappa \colon H^{*}_{{S^1}}(\BLS(M)) \rightarrow H^{*}(\widetilde{\M_{0}})$, so that 
\[
\mathcal{K} \colon H^{*}_{{S^1}}(M) \overset{\beta^{*}}{\longrightarrow} H^{*}_{{S^1}}(\BLS(M))\overset{\kappa}{\longrightarrow} H^{*}(\widetilde{\M_{0}}).
\]
\end{rem}

\subsection{The surjectivity theorem}

An immediate question is whether the resolution Kirwan map is non-trivial; more precisely, how large its image is. To study this question, recall from Section \ref{sec:resdiff} that there is a natural map  $$H^{*}(\M_{0};\R)\to  H^{*}(\widetilde{\Omega}^{*}(\M_{0}),d)$$ from the singular cohomology $H^{*}(\M_{0};\R)$ with real coefficients\footnote{In this section we write singular cohomologies always with the explicit coefficient ring $\R$ in order to distinguish them notationally from the other considered cohomologies.} to the resolution cohomology, the target space of $\mathcal K$. As noted in Remark \ref{rem:notinj}, this natural map  is not necessarily injective, but we have seen in Corollary \ref{generalF} that it is at least injective in even degrees under appropriate assumptions. This shows that the image  of $H^{*}(\M_{0};\R)$ in $H^{*}(\widetilde{\Omega}^{*}(\M_{0}),d)$ is an interesting space, even if in general it contains less information than the full singular cohomology $H^{*}(\M_{0};\R)$. The following main result of this section can therefore be seen as a form of Kirwan surjectivity in the singular setting:
\begin{thm}\label{surjectivity}The image of the resolution Kirwan map $\mathcal{K} \colon H^{\ast}_{{S^1}}(M) \rightarrow H^{*}(\widetilde{\Omega}^{*}(\M_{0}),d)$ contains the image of the natural map $H^{*}(\M_{0};\R)\to  H^{*}(\widetilde{\Omega}^{*}(\M_{0}),d)$.
\end{thm}
In order to prove  Theorem \ref{surjectivity} we need some preparations. We want to prove Theorem \ref{surjectivity} by evoking the equivariant de Rham isomorphism on the one hand and surjectivity of
\[
\iota^{*} \colon H^{*}_{{S^1}}(M) \rightarrow H^{*}_{{S^1}}(J^{-1}(0)),
\]
on the other hand, which was proved by Kirwan using Morse--Bott--Kirwan theory  \cite[Theorem 8.1]{HK16}. Note that for any homomorphism $\Psi$ between the cohomologies of two cochain complexes of $\mathbb{R}$-vector spaces there exists a cochain map  inducing $\Psi$, which follows from the fact that every short exact sequence of $\mathbb{R}$-vector spaces splits (c.f.\ \cite{weibel94}[Exercise 1.1.3 and beginning of Section 1.4]). Thus, the equivariant de Rham isomorphism
\[
\Psi_{dR}^{{S^1}} \colon H^{*}_{{S^1}}(M) \rightarrow H^{*}_{{S^1}}(M;\R)
\]
is induced by a cochain map
\[
\Psi_{dR}^{{S^1}} \colon C_{{S^1}}(M) \rightarrow S_{{S^1}}(M;\R),
\]
which we denote by the same letter. Here $S_{{S^1}}(-;\R)$ denotes the cochain complex of singular equivariant cochains with real coefficients.
For some $\omega=\iota_{\top}^{*}\eta\in \Omega^{*}(\J^{-1}(0))$ we set 
\[
\varphi(\omega):=\iota^{*}(\Psi_{dR}^{{S^1}}\eta)\in S_{{S^1}}(M;\R).
\]
This map is a well-defined cochain map since the following diagram
\begin{center}
\begin{tikzcd}
C_{{S^1}}(M) \arrow{d}{\iota_{\top}^{*}} \arrow{r}{\Psi_{dR}^{{S^1}}} &S_{{S^1}}(M;\R) \arrow{d}{\iota^{*}}\\
C_{{S^1}}(\Omega(J^{-1}(0))) \arrow{d} \arrow{r}{\varphi} &S_{{S^1}}(J^{-1}(0);\R) \arrow{d}{\mathrm{res}}\\
C_{{S^1}}(J^{-1}(0)^{\top})  \arrow{r}{\Psi_{dR}^{{S^1}}} &S_{{S^1}}(J^{-1}(0)^{\top};\R)
\end{tikzcd}
\end{center}
commutes and $J^{-1}(0)^{\top}$ is dense in $J^{-1}(0)$, and we obtain a well-defined induced map
\[
\varphi \colon  H^{*}_{{S^1}}(J^{-1}(0)) \rightarrow H^{*}_{{S^1}}(J^{-1}(0);\R).
\]
By commutativity of
\begin{center}
\begin{tikzcd}
H^{*}_{{S^1}}(M) \arrow{d}{\iota_{\top}^{*}} \arrow{r}{\Psi_{dR}^{{S^1}}} &H^{*}_{{S^1}}(M;\R) \arrow[two heads]{d}{\iota^{*}}\\
H^{*}_{{S^1}}(J^{-1}(0)) \arrow{r}{\varphi} &H^{*}_{{S^1}}(J^{-1}(0);\R) 
\end{tikzcd}
\end{center}
one sees that $\varphi$ is surjective. Next, we want to relate the cohomology of the equivariant resolution forms to the cohomology of the strict transform $\widetilde{\J^{-1}(0)}$. Analogously to Proposition \ref{resolutionforms} we have 
\begin{lem}\label{EquivTransform}
The map $\Phi$ defined as composition
\begin{multline*}
\Phi \colon H_{{S^1}}^{*}(\widetilde{\J^{-1}(0)};\mathbb{R}) \overset{(\Psi_{dR}^{{S^1}})^{-1}}{\rightarrow}  H_{{S^1}}^{*}(\widetilde{\J^{-1}(0)}) \overset{(\widetilde{\iota}_{\top})^{*}}{\rightarrow}  H_{{S^1}}^{*}(\beta^{-1}(J^{-1}(0)^{\top})) \\\overset{(\beta_{\top}^{*})^{-1}}{\longrightarrow}  H_{{S^1}}^{*}(C_{{S^1}}(\widetilde{\Omega}(J^{-1}(0))),d_{{S^1}})
\end{multline*}
is an isomorphism.
\end{lem}
\begin{proof}
The inverse is given by
\[
\omega \mapsto \beta_{\top}^{*}\omega=(\iota'_{\top})^{*}\widetilde{\eta} \mapsto (\iota')^{*}\widetilde{\eta} \mapsto \Psi_{dR}^{{S^1}}((\iota')^{*}\widetilde{\eta} ),
\]
which is well defined because $\beta^{-1}(J^{-1}(0)^{\top})$ is dense in $\widetilde{\J^{-1}(0)}$ and the strict transform $\widetilde{\J^{-1}(0)}$ is a ${S^1}$-invariant submanifold of $\BLS(M)$.
\end{proof}

\begin{lem}
The pull-back $\pi_{\top}^{*} \colon \widetilde{\Omega}^{*}(\M_{0}) \rightarrow \widetilde{\Omega}^{*}\left(\J^{-1}(0)\right)_{\mathrm{bas} \, \mathfrak{g}}$ is an isomorphism of real algebras.
\end{lem}
\begin{proof}
Since the ${S^1}$-action is free on $\J^{-1}(0)^{\top}$, the pull-back 
\[
\pi_{\top}^{*} \colon \Omega(\M_{0}^{\top}) \rightarrow \Omega^{*}(\J^{-1}(0)^{\top})_{\mathrm{bas} \, \mathfrak{g}}
\]
is an isomorphism by \cite[Proposition 2.5]{goertscheszoller}. Now, let $\omega \in \widetilde{\Omega}^{*}(\M_{0})$. Then there is $\widetilde{\eta} \in \Omega^{*}(\BLS(M))$ such that $\beta_{\top}^{*}\pi_{\top}^{*}\omega=(\iota'_{\top})^{*}\widetilde{\eta}$. Thus, $\pi_{\top}^{*}\omega$ is an element of $\widetilde{\Omega}^{*}\left(\J^{-1}(0)\right)_{\mathrm{bas} \, \mathfrak{g}}$. Analogously $(\pi_{\top}^{*})^{-1}\theta \in \widetilde{\Omega}^{*}(\M_{0})$ for any $\theta \in \Omega^{*}(\J^{-1}(0)))_{\mathrm{bas} \, \mathfrak{g}}$.
\end{proof}

Now we are ready to prove Theorem \ref{surjectivity}:

\begin{proof}[Proof of Theorem \ref{surjectivity}]
Consider the commutative diagram
\begin{center}
\begin{tikzcd}
H^{\ast}_{{S^1}}(M;\R) \arrow[two heads]{r}{\iota^{*}} &H^{\ast}_{{S^1}}(J^{-1}(0);\R) \arrow{r}{\beta^{*}} &H^{\ast}_{{S^1}}(\widetilde{\J^{-1}(0)};\R) \arrow{d}{\Phi} \\
H^{\ast}_{{S^1}}(M) \arrow{u}{\Psi_{dR}^{{S^1}}}  \arrow{ddrr}{\mathcal{K}} \arrow{r}{\iota_{\top}^{*}} &H^{\ast}_{{S^1}}(J^{-1}(0)) \arrow{u}{\varphi} \arrow{r}{\mathrm{inc}} &H^{\ast}\left(C_{S^1}\big (\widetilde{\Omega}^{*}\left(\J^{-1}(0)\big )\right),d_{{S^1}}\right) \arrow{d}{\Car}\\
& &H^{\ast}\left(\widetilde{\Omega}^{*}\left(\J^{-1}(0)\right)_{\mathrm{bas} \, \mathfrak{g}},d\right) \arrow{d}{\left(\pi_{\top}^{*}\right)^{-1}}\\  
 &  &H^{\ast}(\widetilde{\Omega}^{*}(\M_{0}),d) 
\end{tikzcd}
\end{center}
where the maps $\iota$, $\iota'$, and $\beta$ form the commutative diagram
\begin{center}
\begin{tikzcd}
\BLS(M) \arrow{d}{\beta} &\widetilde{\J^{-1}(0)} \arrow{d}{\beta} \arrow{l}{\iota'} \\
M &J^{-1}(0) \arrow{l}{\iota}
\end{tikzcd}
\end{center}
Denote by denote by $\mathbb H\subset H^{*}(\widetilde{\Omega}^{*}(\M_{0}),d)$  the image of $H^{*}(\M_{0};\R)$  under the natural injection and let $[\rho] \in \mathbb H $ with $\pi_{\top}^{*}\rho=:\omega=\iota_{\top}^{*}\eta$ for some $\eta \in \Omega(M)$. Then $\mathrm{Car}([1\otimes \omega])=[\omega]$. Since $\iota_{\top}^{*}(1 \otimes \eta)=1 \otimes \omega$, we can regard
\[
[1 \otimes \omega] \in H^{\ast}_{{S^1}}(J^{-1}(0))
\]
and $\varphi([1 \otimes \omega]) \in H^{\ast}_{{S^1}}(J^{-1}(0);\R)$. By Kirwan's surjectivity theorem \cite[Theorem 8.1]{HK16} there is $[\widetilde{\eta}'] \in H^{\ast}_{{S^1}}(M;\R)$ with
\[
\iota^{*}[\widetilde{\eta}']=\varphi([1 \otimes \omega]).
\]
By commutativity and setting $[\eta']:=(\Psi_{dR}^{{S^1}})^{-1}[\widetilde{\eta}']$, we have 
\[
\varphi([1 \otimes \omega])=\iota^{*}\Psi_{dR}^{{S^1}}((\Psi_{dR}^{{S^1}})^{-1}[\widetilde{\eta}'])=\varphi([\iota_{\top}^{*}\eta']).
\]
But this implies that in $H^{\ast}\big(C_{S^1}\big (\widetilde{\Omega}^{*}\left(\J^{-1}(0)\big )\right),d_{{S^1}}\big)$ we have 
\begin{align*}
[1 \otimes \omega]&=\mathrm{inc}([1\otimes \omega])= \Phi(\beta^{*}(\varphi([1 \otimes \omega])))\\
&=\Phi(\beta^{*}(\varphi([\iota_{\top}^{*}\eta'])))=\mathrm{inc}(\iota_{\top}^{*}[\eta']).
\end{align*}
In total, we conclude
\[
\mathcal{K}([\eta'])=[\rho].
\]
\end{proof}

\subsection{Consequences and remarks} Les us collect some consequences of the surjectivity theorem proved in the previous subsection, followed by some remarks.
 Suppose that a subspace $\mathcal H\subset H^{*}(\M_{0};\R)$ is \emph{injectively} mapped into $H^{*}(\widetilde{\Omega}^{*}(\M_{0}),d)$ by the natural map $H^{*}(\M_{0};\R)\to  H^{*}(\widetilde{\Omega}^{*}(\M_{0}),d)$ and denote by $\mathbb H\subset H^{*}(\widetilde{\Omega}^{*}(\M_{0}),d)$  the image of $\mathcal H$. Then Theorem \ref{surjectivity} allows us to build from the resolution Kirwan map $\mathcal K$ a \emph{surjective} map  
\bq
H^{*}_{{S^1}}(M)\owns {\mathcal K}^{-1}(\mathbb H)\to \mathcal H\label{eq:surjmap}
\eq by identifying $\mathcal H$ with $\mathbb H$ and composing $\mathcal K$ with a linear projection $\pr_\mathbb{H}: H^{*}(\widetilde{\Omega}^{*}(\M_{0}),d)\to \mathbb{H}$. As an instance of this, we have the following

\begin{cor} \label{ex:even}Assume that the cohomology groups $H^{2k+1}(F)$ vanish for all connected components $F \subset M^{S^{1}}\cap \J^{-1}(0)$.  Then, $\mathcal{K}$ induces a surjective map 
\bq
\check \kappa: H^{\mathrm{ev}}_{{S^1}}(M)\to H^{\mathrm{ev}}(\M_{0};\R).\label{eq:surjmap2}
\eq
In the case that the odd cohomology of $\M_{0}$ vanishes, we obtain a surjective linear map
\begin{equation}
\check \kappa \colon H^{*}_{{S^1}}(M)\to H^{*}(\M_{0};\R).\label{eq:surjmap3}
\end{equation}
\end{cor}
\begin{proof}
Identify the  cohomology in even degrees  $H^{\mathrm{ev}}(\M_{0};\R)$, which by Corollary \ref{generalF}  is injected into $H^{*}(\widetilde{\Omega}^{*}(\M_{0}),d)$ by the natural map, with its image $\mathbb{H}\subset H^{*}(\widetilde{\Omega}^{*}(\M_{0}),d)$. Choose further a linear projection $\pr_{\mathbb{H}}: H^{*}(\widetilde{\Omega}^{*}(\M_{0}),d)\to \mathbb{H}$. Since $\mathcal{K}$ is degree-preserving,  its composition with $\pr_\mathbb{H}$ restricts to a surjective map by Theorem \ref{surjectivity}. 
\end{proof}

\begin{rem}\label{rem:18.08.2023}
The resolution Kirwan map $\mathcal{K} \colon H^{*}_{{S^1}}(M) \rightarrow H^{*}(\widetilde{\M_{0}})$ can not be surjective. Indeed, as noted in Remark \ref{rem:15.07.2023},  it can be seen as the composition of the map induced by the blow-down $\beta$ composed with the regular Kirwan map of the blown-up Hamiltonian action. Now, by \cite[Proposition 2.4]{mcduff1987} the cohomology of the blow-up $\BLS(M)$ is generated by cohomology classes pulled back from $M$ and the symplectic class of each exceptional bundle. These exceptional symplectic classes are not in the image of $\beta^{*}$. In particular, the symplectic class of $\widetilde{\M_{0}}$ is not in the image of $\mathcal{K}$.
\end{rem}

\begin{rem}
As already explained, in the case that the odd cohomology of all $F \in \mathcal{F}_{0}$ and the odd cohomology of $\M_{0}$ vanishes, we have a linear surjection (\ref{eq:surjmap3}). It would be interesting to know if $\check \kappa$ is in fact a ring homomorphism, which is not clear from our construction{. For example, in the case  of isolated fixed points and $H^{\mathrm{odd}}(\M_{0})=0$, the isomorphism
\[
H^{*}(\widetilde{\M}_{0})\cong H^{*}\left(\M_{0}\right) \oplus  \bigoplus\limits_{{F \in \mathcal{F}_{0}}} {\frac{\mathbb{R}[\widetilde{\sigma}_{0}|_{\widetilde F},\widetilde{\Xi}|_{\widetilde F}]_{\geq 1}}{I_F}}
\]
from Corollary \ref{splitting} is} linear but not necessarily multiplicative. 
\end{rem}

\section{Intersection cohomology and its relation to  singular cohomology}\label{sec:inters} 

In this section, we would like to discuss our results in the light  of  the  intersection cohomology of  $\M_0$, regarded as  a {simple stratified space}. Let us therefore briefly recall this cohomology theory following   the presentation of Lerman-Tolman in \cite[Section 2]{lerman-tolman00}. Consider a general  \emph{simple stratified space} ${\bf X}$, that is a topological Hausdorff space 
 \bqn
 {\bf X}= {\bf X}_\mathrm{top} \sqcup \bigsqcup_i {\bf Y}_i
 \eqn 
 consisting of  a disjoint union of even-dimensional orbifolds called \emph{strata}. Among these strata there exists an  open and dense stratum ${\bf X}_\mathrm{top}$ called the \emph{top stratum}. The remaining strata ${\bf Y}_i$ are connected and called the \emph{singular strata}.  Each singular stratum ${\bf Y}_i$ has a neighborhood ${\bf T}_i\subset {\bf X}$ and a retraction map $r_i:{\bf T}_i \rightarrow {\bf Y}_i$ that is a $C^0$ fiber bundle with fiber given by an open cone over a suitable orbifold ${\bf L}_i$, such that  $r_i:{\bf T}_i\setminus {\bf Y}_i \rightarrow {\bf Y}_i$ is a $C^\infty$ fiber bundle of orbifolds with fiber ${\bf L}_i \times (0,1)$.   Next, let $r:{\bf E} \rightarrow {\bf B}$ be a smooth submersion of orbifolds.  One then defines the \emph{Cartan filtration} 
 \bqn 
 \Fbb_0\Omega^\ast({\bf E}) \subset \Fbb_1\Omega^\ast({\bf E}) \subset \Fbb_2\Omega^\ast({\bf E}) \subset \dots...
 \eqn
  of the complex $\Omega^\ast({\bf E})$ of differential forms on ${\bf E}$ by setting
     \begin{align*} 
 \Fbb_k\Omega^\ast({\bf E}):=&\Big \{ \omega \in \Omega^\ast({\bf E})  \mid i_{v_0} i_{v_1} \dots i_{v_k} \omega_e=0, \, i_{v_0} i_{v_1} \dots i_{v_k}  d\omega_e=0  \\ &\text{ for all } e \in E \text{ and } v_0,\dots,v_k \in \ker dr_e\Big \}, 
 \end{align*}
where by convention one sets $i_{v_0} i_{v_1} \dots i_{v_k}  \omega=0$ if $\deg \omega \leq k$. Note that $ \Fbb_0\Omega^\ast({\bf E})$ consists of basic forms, $ \Fbb_1\Omega^\ast({\bf E})$ of forms which are basic after one contraction, and so forth. Now, consider a simple stratified space $({\bf X}, r_i, {\bf T}_i, {\bf Y}_i)$ as above and let $\bar p:\mklm{{\bf Y}_i}\rightarrow \N$ be a function called  \emph{perversity} in this context. Then the  complex of \emph{intersection differential forms of perversity $\bar p$}  is the subcomplex
\begin{multline*}
I\Omega^\ast _{\bar p}({\bf X}):=\Big \{ \omega \in \Omega^\ast({\bf X}_\mathrm{top}) \mid \omega|_{{\bf U}_i} \in \Fbb_{\bar p({\bf Y}_i)} \Omega^\ast ({\bf U}_i\cap {\bf X}_\mathrm{top}) \text{  for some neighborhood }\\
 {\bf U}_i \subset {\bf T}_i \text{  of ${\bf Y}_i$ for all } i \Big \},
\end{multline*}
where the filtration is defined with respect to the submersions $r_{i|{\bf U}_i}:{\bf U}_i \setminus {\bf Y}_i \rightarrow {\bf Y}_i$, and the differential is the usual exterior {derivative} $d$. The \emph{intersection cohomology $IH^\ast _{\bar p}({\bf X})$ with perversity $\bar p$} of the stratified space ${\bf X}$ is the cohomology of the complex $(I\Omega^\ast _{\bar p}({\bf X}),d)$. 
{ Similarly, we introduce the subcomplex 
\begin{align*}
I\Omega^\ast _{\bar p}({\bf X})_c:=\Big \{ \omega \in I\Omega^\ast _{\bar p}({\bf X}) \mid \supp \omega \subset {\bf Z} \text{ for some compact set } {\bf Z}\subset {\bf X}  \Big \},
\end{align*}
of \emph{intersection differential forms of perversity $\bar p$ and compact support} and its cohomology $IH^\ast _{\bar p}({\bf X})_c$. Now, note  that if $q > \dim {\bf Y}_i +\bar p({\bf Y}_i)$, every $\omega \in I\Omega^q _{\bar p}({\bf X})$ vanishes near ${\bf Y}_i$. In particular, if $\dim {\bf X}_\mathrm{top} > \dim {\bf Y}_i +\bar p({\bf Y}_i)$ for all singular strata, every $\omega \in I\Omega^{\dim {\bf X}_\mathrm{top}}_{\bar p}({\bf X})_c$ has compact support in ${\bf X}_\mathrm{top}$.  Consequently, one can integrate such a form over the top stratum, since the latter is oriented. If even  $\dim {\bf X}_\mathrm{top}-1 > \dim {\bf Y}_i +\bar p({\bf Y}_i)$ for all $i$, one obtains a well-defined map 
\bqn 
\int: IH^{\dim {\bf X}_\mathrm{top}}_{\bar p}({\bf X})_c \longrightarrow \R,
\eqn
which is extended by zero to $IH^\ast_{\bar p}({\bf X})_c$. Now, {the \emph{middle perversity} is defined as $\bar m({\bf Y}_i):=(\dim {\bf X}_\mathrm{top} - \dim {\bf Y}_i)/2 -1$. G}iven $\alpha, \beta \in  IH^\ast_{\bar m}({\bf X})$ one has $\alpha \wedge \beta \in IH^\ast_{2\bar m}({\bf X})$. Since by definition of the middle perversity $\dim {\bf X}_\mathrm{top} -1 > \dim {\bf Y}_i +2 \bar m({\bf Y}_i)$ for all $i$, there is a well-defined pairing 
\bqn 
\eklm{\cdot,\cdot}:IH^p_{\bar m}({\bf X})\otimes  IH^q_{\bar m}({\bf X})_c \longrightarrow \R, \qquad [\alpha] \otimes [\beta] \longmapsto \int_{{\bf X}_\mathrm{top}} \alpha \wedge \beta,
\eqn
which is non-degenerate by the usual Poincar\'e duality for oriented manifolds. }

Next, let $\widetilde {\bf X}$ be a resolution of a simple stratified space $\bf X$. 
In \cite[Definition 5.2]{lerman-tolman00}, intersection resolution forms of middle perversity $\bar m$ were introduced as intersection differential forms on ${\bf X}_\mathrm{top}$ that extend to a globally defined differential form on $\widetilde {\bf X}$. As it turns out, the cohomology of the complex of intersection resolution forms is isomorphic to $IH^\ast_{\bar m}({\bf X})$.  

In our context, let us assume that $S^1$ acts semi-freely on the zero level of the momentum map, that is, $S^1$ acts freely on $\J^{-1}(0)\setminus (\J^{-1}(0) \cap M^{S^1})$, so that our orbit type stratification coincides with the infinitesimal orbit type stratification of Lerman-Tolman. As a direct consequence of Corollary \ref{cor:rauisch}  we have
\begin{cor}
The  complex $\widetilde{\Omega}(\M_{0})$ naturally extends the complex of intersection resolution forms on $\M_0$. 
\end{cor}
 In fact, our definition drops the intersection condition using the partial desingularization as a resolution. In the context of  nonsingular, connected, complex, projective varieties and GIT quotients  $IH^{*}_{\bar m}(\M_{0})$ can even be identified as a canonical summand in $ H^{*}(\widetilde{\M_{0}})$ by using algebro-geometric arguments involving the decomposition theorem of Beilinson--Bernstein--Deligne and the Hard Lefschetz theorem, see \cite[pp. 200 and 234]{JKKW03}, but in general such a canonical identification is not clear. In any case, in view of Corollary \ref{splitting} our approach opens up the possibility  to relate $H^\ast(\M_0;\R)$ to   $IH^\ast (\M_0)$ via maps to $H^\ast(\widetilde \M_0)$. Notice further that in \cite{kiem-woolf2005},  Kiem and Woolf identified the intersection cohomology $IH^{*}_{\bar m} (\M_{0})$ with respect to the middle perversity with a subspace of $H^{*}_{G}(\J^{-1}(0))$. Namely, they found an isomorphism
\[
IH^{*}(\M_{0}) \simeq \left\{ \eta \in H^{*}_{G}(\J^{-1}(0)) \mid \eta|_{F} \in H^{*}(F)\otimes \mathbb{R}[x]_{\leq 2d-2}  \, \forall F \subset \J^{-1}(0)  \right\},
\]
where $H^{*}(F)\otimes \mathbb{R}[x]_{\leq 2d_{i}-2} \subset  H^{*}(F)\otimes \mathbb{R}[x] \simeq H^{*}_{G}(F)$ and $d:= \min\{l_{F}^{+},l_{F}^{-}\}$, which identifies the intersection pairing in $IH^{*}(\M_{0})$ with the cup product in $H^{*}_{G}(\J^{-1}(0))$. In the same spirit, Remark \ref{singularinresolution} provides a description of singular cohomology inside resolution cohomology as it characterizes singular classes as those classes in resolution cohomology whose restriction to the exceptional fibers vanish in cohomology.

Finally, if we compare the resolution Kirwan map introduced in Defintion \ref{def:20.04.2022} with the intersection Kirwan map $\kappa_{IH}$ from \eqref{eq:18.08.2023}, it is clear that they differ  by the projection onto $IH^\ast(\M_0)$.  Furthermore, the fact that  the maps \eqref{eq:surjmap}, \eqref{eq:surjmap2} are  non-canonical, compare Remark \ref{canonicity}, is in congruence with   non-canonicity phenomena occurring in other approaches to a resolution Kirwan map, see \cite[Theorem 1 and Theorem 6]{kiem-woolf2006}, \cite[Corollary 3.5]{woolf2003} and \cite[p. 234]{JKKW03}. Only in the GIT case of \cite{JKKW03} already mentioned above  there is a natural choice for the projection $H^{*}(\widetilde{\M}_{0})\simeq IH_{\bar m}^{*}(\widetilde{\M}_{0}) \rightarrow IH_{\bar m}^{*}(\M_{0})$ due to the Hard Lefschetz Theorem.

\section{Abelian Polygon Spaces}\label{sec:17.08.2023}

To close this paper, let us discuss one interesting family of examples where our theory applies, namely abelian polygon spaces. They  have been of interest in symplectic and algebraic geometry as well as in combinatorics for a long time, see \cite{hausmann-knutson1998} and the abundant references therein. They play a role in studying the moduli space of $n$-times punctured Riemann spheres as well as the moduli space of $n$ unordered weighted points in $\CP^{1}$;  in differential geometry, polygon spaces spark interest as they are connected to the moduli space of flat connections on Riemann-surfaces, see \cite{agapito-godinho2009}. 

Generalizing Example \ref{2spheres}, we consider first a product of $2$-spheres 
\[
M:=\prod\limits_{i=1}^{n} S^{2}_{r_{i}} \subset \left(\mathbb{R}^{3}\right)^{n}
\]
of radii $r_i>0$ endowed with the product symplectic form, which it inherits from the symplectic forms $\sigma \in \Omega^{2}(S^{2}_{r_i})$ and let $S^{1}$ act on this product diagonally, where the circle acts on each $2$-sphere by standard rotation around the $z$-axis. This action is Hamiltonian and a momentum map is given by 
\begin{align*}
\J \colon \prod\limits_{i=1}^{n} S^{2}_{r_i} &\longrightarrow \mathbb{R}, \qquad (x_{i},y_{i},z_{i})_{i} \longmapsto \sum\limits_{i=1}^{n}z_{i}.
\end{align*}
In case $(r_{1},\ldots,r_{n})=(1,\ldots,1)$ the regularity of $0 \in \mathbb{R}$ depends on the parity of $n$ and $0$ is a regular value of the momentum map if and only if $n$ is odd. So for our purposes we will assume that $n$ is even and deal with the stratified symplectic quotient
\[
\M_{0}:=\frac{\J^{-1}(0)}{S^{1}}=:M/\!\!/_{0}{S^1}.
\]
The singular stratum of this space consists of isolated points which are induced by the fixed point set 
\[
\J^{-1}(0)\cap \left( \prod\limits_{i=1}^{n} S^{2}_{r_i} \right)^{S^{1}} = \left\{((0,0,s_1 r_{1}),\ldots,(0,0,s_n r_{n})) \in \J^{-1}(0)\,|\, s_j\in \{1,-1\}\right\}.
\] 
Applying our results in this setting, in particular Corollary \ref{ex:even},  then shows  that there is a surjective map
\[
 H^{\mathrm{ev}}_{S^{1}}\left(\prod\limits_{i=1}^{n} S^{2}_{r_i} \right) \longrightarrow H^{\mathrm{ev}}(\M_{0};\mathbb{R}).
\]
Because the fixed point set $\left( \prod\limits_{i=1}^{n} S^{2}_{r_i} \right)^{S^{1}}$ consists of isolated points, our results will also hold for general values $C \in \mathbb{R}$ of the momentum map, since we can shift the momentum map by any additive constant, $S^{1}$ being abelian, and we denote the symplectic quotient at the value $C$ by $M /\!\!/_{C} {S^1}:=\frac{\J^{-1}(C)}{{S^1}}$. Relying on this fact we shall now apply our results to abelian polygon spaces. To introduce them, let us first introduce polygon spaces.

\begin{definition}
The polygon space $\mathrm{Pol}(r_{1},\ldots,r_{n})$ is defined as the family of equivalence classes of piecewise linear paths in $\mathbb{R}^{3}$ consisting of $n$ linear segments, the $i$-th of which is of length $r_i$, such that the $n$-th segment ends where the $1$st begins, and two paths are considered equivalent if they are congruent, that is, equal up to a finite number of rotations and translations. In other words:
\[
\mathrm{Pol}(r_{1},\ldots,r_{n}):=\left\{(x_{1},\ldots,x_{n}) \in \prod\limits_{i=1}^{n} S^{2}_{r_i} \Big| \sum\limits_{i=1}^{n}x_{i}=0 \right\} \Big/ \SO(3),
\] 
where $\SO(3)$ acts diagonally by the restriction of the standard action on $\R^3$.
\end{definition}

These spaces have a natural symplectic structure as they arise as symplectic reductions of $\prod\limits_{i=1}^{n} S^{2}_{r_i}$. This is because
\begin{align*}
\prod\limits_{i=1}^{n} S^{2}_{r_i} \longrightarrow \mathbb{R}^{3}, \qquad (x_{1},\ldots,x_{n})& \longmapsto  \sum\limits_{i=1}^{n}x_{i}\in \R^3\cong \mathrm{so}(3)^\ast
\end{align*}
is a momentum map of the $\SO(3)$-action and 
\[
\mathrm{Pol}(r_{1},\ldots,r_{n})=\left(\prod\limits_{i=1}^{n} S^{2}_{r_i} \right) \Big/ \!\! \Big/_{0}\, \SO(3).
\]
Consequently,  symplectic geometry is a powerful tool to study polygon spaces. One important tool to determine the cohomology ring of smooth polygon spaces in \cite{hausmann-knutson1998} are abelian polygon spaces, which we now introduce.

\begin{definition}
The abelian polygon space $\mathrm{APol}(r_{1},\ldots,r_{n})$ is defined as 
\[
\mathrm{APol}(r_{1},\ldots,r_{n})=\left(\prod\limits_{i=1}^{n-1} S^{2}_{r_i} \right) \Big/\!\! \Big/_{r_{n}} S^{1}.
\]
\end{definition}

While the homology of singular polygon spaces has been studied in \cite{kamiyama1998}, where it was shown that the homology groups of singular polygon spaces do not fulfill Poincar\'e-duality \cite[Remark 1.7]{kamiyama1998}, the case of singular abelian polygon spaces seems to be uncovered in  literature. But it is not difficult to see that they carry serious singularities as well. Indeed, computing the local homology of a singular point $p$, which is induced by a fixed point consisting of $\ell^{+}$ north poles $N$ of the form $N=(0,0,r_{i})$ and $\ell^{-}$ south poles $S$ of the form $S=(0,0,-r_{j})$, reveals that 
\begin{align*}
&H_{*}(\mathrm{APol}(r_{1},\ldots,r_{n}),\mathrm{APol}(r_{1},\ldots,r_{n})\setminus \{p\})\\
&\qquad\qquad =H_{*}\Big (C(S^{2\ell^{+}-1}\times_{S^{1}}S^{2\ell^{-}-1}),\mathbb{R}\times \frac{S^{2\ell^{+}-1}\times S^{2\ell^{-}-1}}{S^{1}}\Big )\\
&\qquad\qquad\simeq \widetilde{H}_{*}\Big (\mathbb{R}\times \frac{S^{2\ell^{+}-1}\times S^{2\ell^{-}-1}}{S^{1}}\Big ),
\end{align*}
where we made use of excision and the local normal form theorem  as in \cite[p. 231]{hatcher02}. Here  $\widetilde{H}_{*}$ stands for the reduced singular homology with real coefficients, $C(X)$ denotes the cone construction, which associates to a topological space $X$ the cone $(X \times [0,1))/\sim$,  where $(x,t) \sim (y,s)$ iff $t=s=0$,  and the local normal form theorem was applied in the same way as in \cite[Proposition 3.3]{lerman-tolman00}. Consequently, abelian polygon spaces are neither homology manifolds nor orbifolds. 

In contrast to the regular situation, polygon spaces  need not be even cohomology spaces, that is, there exists a $k$ such that $H^{2k-1}(\mathrm{Pol}(r_{1},\ldots,r_{n}))\neq 0$, as was pointed out in \cite[Remark 1.7]{kamiyama1998}. Furthermore,  the cohomology ring $H^{*}(\mathrm{Pol}(r_{1},\ldots,r_{n}))$ remains unknown in general. Nevertheless, as a consequence of our results applied to abelian polygon spaces one obtains

\begin{thm}\label{thm:27.11.2023}
There is a linear surjection
 \[
 \check \kappa : H^{2k}_{S^{1}}\left( \prod\limits_{i=1}^{n-1} S^{2}_{r_i} \right) \longrightarrow H^{2k}(\mathrm{APol}(r_{1},\ldots,r_{n})).
 \]
\end{thm}

\begin{proof} This follows directly from Corollary \ref{ex:even}.
\end{proof}

We conclude by exploring one example, related to the ones above, which shows that  the singular cohomology of a singular symplectic quotient by a circle action might not vanish in odd degrees.

\begin{example}\label{4spheres}
Consider as $M$ the product of four $2$-spheres of radius $1$ with momentum map
\begin{align*}
\J \colon S^{2} \times S^{2} \times S^{2} \times S^{2} &\longrightarrow \mathbb{R}, \qquad (x_{i},y_{i},z_{i})_{1 \leq i \leq 4} \longmapsto z_{1}+z_{2}+z_{3}+z_{4}.
\end{align*} 
The image of this map is the closed interval $[-4,4]$ and the critical values inside this interval are $\{-4,-2,0,2,4\}$ as illustrated in the diagram

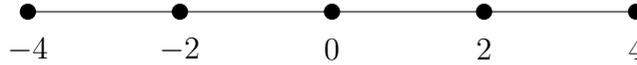
\begin{center}
\begin{tikzpicture}
\draw (-4,0)--(4,0);

\filldraw[black] (-4,0) circle (0.1cm);
\filldraw[black] (-2,0) circle (0.1cm);
\filldraw[black] (0,0) circle (0.1cm);
\filldraw[black] (2,0) circle (0.1cm);
\filldraw[black] (4,0) circle (0.1cm);

\node[] at (-4,-0.5) {$-4$};
\node[] at (-2,-0.5) {$-2$};
\node[] at (0,-0.5) {$0$};
\node[] at (2,-0.5) {$2$};
\node[] at (4,-0.5) {$4$};

\end{tikzpicture}
\end{center}

 \noindent If we intersect the level sets of these critical values with the fixed point set of the action, we obtain with the notation $N:=(0,0,1),S:=(0,0,-1) \in S^{2}$ the equalities 
 \begin{align*}
\J^{-1}(4) \cap M^{S^{1}}&= \{(N,N,N,N)\} , \; \#\left( \J^{-1}(4) \cap M^{S^{1}} \right)=1,\\
\J^{-1}(2) \cap M^{S^{1}}&= \{x \in \{N,S\}^{4} \mid \text{exactly three entries of $x$ are $N$}\} , \; \#\left( \J^{-1}(2) \cap M^{S^{1}} \right)=4,\\
\J^{-1}(0) \cap M^{S^{1}}&= \{x \in \{N,S\}^{4} \mid \text{exactly two entries of $x$ are $N$}\} , \; \#\left( \J^{-1}(0) \cap M^{S^{1}} \right)=6,\\
\J^{-1}(-2) \cap M^{S^{1}}&=\{x \in \{N,S\}^{4} \mid \text{exactly one entry of $x$ is $N$}\} , \; \#\left( \J^{-1}(-2) \cap M^{S^{1}} \right)=4,\\
\J^{-1}(-4) \cap M^{S^{1}}&= \{(S,S,S,S)\}, \; \#\left( \J^{-1}(-4) \cap M^{S^{1}} \right)=1.
\end{align*}
We now look at the symplectic quotients $\M_{\varepsilon}$ where $\varepsilon \in [-4,4]$ and study them by invoking the local normal form theorem, the Duistermaat--Heckman theorem, and \cite[Section 2.3]{guillemin94}. As the $(-4)$-level only consists of the isolated fixed point $(S,S,S,S)$, the local normal form theorem  from Proposition \ref{prop:localnormform} implies that  
\[
\M_{\varepsilon} \cong \CP^{3} 
\]
for $\varepsilon>-4$ but close to $-4$, where the isomorphism is meant as a diffeomorphism. By the Duistermaat--Heckman theorem it follows that $\M_{\varepsilon} \cong \CP^{3}$   for all $\varepsilon \in (-4,-2)$. When we cross the critical value $-2$, as explained in \cite[pp.\ 35]{guillemin94}, we are facing four fixed points where the positive weight space in the isotropy representation in each of these fixed points is complex one-dimensional. Therefore $\M_{\varepsilon}$ is diffeomorphic to the blow-up of $\CP^{3}$ in four points when $\varepsilon \in (-2,0)$. In particular, the dimension of the second cohomology group of this space is $\dim H^{2}(\M_{\varepsilon})=5$ by \cite[Proposition 2.4]{mcduff1987}.  Now, when passing from $\varepsilon \in (-2,0)$ to $0$ we have to collapse six $\CP^{1}$s inside $\M_{\varepsilon}$ to points by \cite[p.\ 35]{guillemin94}. For this we consider the decompositions  $\M_{0}=U \cup V$ and $\M_{\varepsilon}=U' \cup V'$, where $U'$ is the union of tubular neighborhoods of the six $\CP^{1}$s, $V'$ is the slighty enlarged complement of $U'$ inside $\M_{\varepsilon}$ and $U$ and $V$ are the corresponding images under the quotient maps defined by collapsing the $\CP^{1}$s to points.  The Mayer-Vietoris sequences connected by the maps induced by collapsing then look like
\begin{center}
\begin{tikzcd}
\ldots \arrow{r} &H^{2}(\M_{0}) \arrow{r} \arrow{d} &H^{2}(U) \oplus H^{2}(V)  \arrow{d} \arrow{r} &H^{2}(U \cap V) \arrow{r} \arrow{d} &H^{3}(\M_{0}) \arrow{r} \arrow{d} &\ldots\\
\ldots \arrow{r} &H^{2}(\M_{\varepsilon}) \arrow{r} &H^{2}(U') \oplus H^{2}(V') \arrow{r} &H^{2}(U' \cap V') \arrow{r} &H^{3}(\M_{\varepsilon}) \arrow{r} &\ldots,
\end{tikzcd}
\end{center}
 In this case $H^{*}(V) \cong H^{*}(V')$, $H^{*}(U\cap V) \cong H^{*}(U' \cap V')$, $H^{*}(U')=\bigoplus H^{*}(\CP^{1})$ and by the long exact cohomology sequence of the pair $(U',\bigcup \CP^{1})$, see \cite[p. 200]{hatcher02}, we have $H^{k}(U)=0$ for $k>1$. Now,  by \cite[Proposition 2.4]{mcduff1987}) one has $H^{3}(\M_{\varepsilon})=0$ and the map
\[
A \colon H^{2}(U') \oplus H^{2}(V') \longrightarrow H^{2}(U' \cap V')
\]
is surjective by exactness. Moreover, as $\dim H^{2}(\M_{\varepsilon})=5$ the kernel of $A$ can be at most $5$-dimensional because it is the image $\Im(H^{2}(\M_{\varepsilon})\rightarrow H^{2}(U') \oplus H^{2}(V'))$. Since $\dim H^{2}(U')=6$, it follows that $A$ is not the zero map. More precisely, there is an element  $(a',0) \in H^{2}(U') \oplus H^{2}(V')$ that  maps to some non-zero $v' \in H^{2}(U' \cap V')$, and $a'$ is not in the image of the map $H^{2}(\M_{\varepsilon}) \rightarrow H^{2}(U')$. Consequently, there is a unique non-zero element $v \in H^{2}(U \cap V)$ which maps to $v'$ as the vertical arrow in the diagram above is an isomorphism.  

We now claim that the image of $v$ under the map $H^{2}(U \cap V) \rightarrow H^{3}(\M_{0})$ is not zero. Indeed, suppose that  it is zero. Then, by  exactness, {and since $H^{2}(U)=0$,} there is an element $(0,w) \in H^{2}(U) \oplus H^{2}(V)$ which maps to $v$. Denote the image of $(0,w)$ under the map $H^{2}(U) \oplus H^{2}(V) \rightarrow H^{2}(U') \oplus H^{2}(V')$ by $(0,w')$.  Since the above diagram is commutative we have that
\[
(a',-w')=(a',0)-(0,w') \longmapsto v'-v'=0.
\]
But this means that there exists an element $b \in H^{2}(\M_{\varepsilon})$ which maps to $(a',-w')$, which is a contradiction, as $a'$ is not in the image of the map $H^{2}(\M_{\varepsilon}) \rightarrow H^{2}(U')$. Thus, we have shown that  $H^3(\M_0;\R)\not=\mklm{0}$.
\end{example}
As a consequence of this example we obtain 
\begin{cor}\label{cor:15.11.2023}
For  a  general Hamiltonian $S^1$-manifold $M$ with singular symplectic quotient $\M_0$ there is no degree-preserving surjection
\[
H^{*}_{{S^1}}(M) \longrightarrow H^{*}(\M_{0};\mathbb{R}).
\]
\end{cor}
\begin{proof}
Consider the situation of Example \ref{4spheres}, where we saw that $H^{3}(\M_{0};\mathbb{R})\neq 0$. On the other hand, the Hamiltonian $S^{1}$-action on $M=S^{2} \times S^{2} \times S^{2} \times S^{2}$  is equivariantly formal, so that 
\[
H^{*}_{S^{1}}(M)=\mathbb{R}[x] \otimes H^{*}(S^{2} \times S^{2} \times S^{2} \times S^{2})
\]
by \cite[Theorem 7.3]{goertscheszoller}.
But the right-hand side is generated in even degrees, and consequently
\[
H^{3}_{S^{1}}(M) =0.
\]
Thus, there cannot be a degree-preserving surjection 
\[
H^{*}_{{S^1}}(M) \longrightarrow H^{*}(\M_{0};\mathbb{R}).
\]
\end{proof}


\providecommand{\bysame}{\leavevmode\hbox to3em{\hrulefill}\thinspace}
\providecommand{\MR}{\relax\ifhmode\unskip\space\fi MR }
\providecommand{\MRhref}[2]{%
  \href{http://www.ams.org/mathscinet-getitem?mr=#1}{#2}
}
\providecommand{\href}[2]{#2}

\end{document}